\documentclass[11pt,reqno]{amsart}

\usepackage{caption,amsmath,amssymb,amsfonts,amsthm,latexsym,graphicx,multirow,color,hyperref,enumerate}
\usepackage[all]{xy}

\numberwithin{table}{section}

\newcommand{\A}{\mathrm{A}} \newcommand{\AGL}{\mathrm{AGL}} \newcommand{\AGaL}{\mathrm{A\Gamma L}}  \newcommand{\Aut}{\mathrm{Aut}}
 \newcommand{\bbF}{\mathbb{F}} \newcommand{\bfO}{\mathbf{O}}
 \newcommand{\Cen}{\mathbf{C}}  \newcommand{\calC}{\mathcal{C}}  \newcommand{\calO}{\mathcal{O}}   \newcommand{\Co}{\mathrm{Co}} 
\newcommand{\D}{\mathrm{D}}

\newcommand{\G}{\mathrm{G}} \newcommand{\GaG}{\mathrm{\Gamma G}} \newcommand{\GaL}{\mathrm{\Gamma L}} \newcommand{\GaSp}{\mathrm{\Gamma Sp}} \newcommand{\GaO}{\mathrm{\Gamma O}} \newcommand{\GaU}{\mathrm{\Gamma U}}  \newcommand{\GL}{\mathrm{GL}} \newcommand{\GO}{\mathrm{O}}  \newcommand{\GU}{\mathrm{GU}}

\newcommand{\lefthat}{\scalebox{1.3}[1]{\text{$\hat{~}$}}}
\newcommand{\M}{\mathrm{M}} \newcommand{\magma}{\textsc{Magma}}
\newcommand{\N}{\mathrm{N}}
\newcommand{\Nor}{\mathbf{N}}

\newcommand{\Pa}{\mathrm{P}}   \newcommand{\PGaL}{\mathrm{P\Gamma L}}     \newcommand{\POm}{\mathrm{P\Omega}} \newcommand{\PSL}{\mathrm{PSL}}    \newcommand{\PSp}{\mathrm{PSp}} \newcommand{\PSU}{\mathrm{PSU}}
\newcommand{\Q}{\mathrm{Q}}
 \newcommand{\Rad}{\mathbf{R}}  
  \newcommand{\SL}{\mathrm{SL}}  \newcommand{\Soc}{\mathrm{Soc}} \newcommand{\Sp}{\mathrm{Sp}}  \newcommand{\SU}{\mathrm{SU}} \newcommand{\Suz}{\mathrm{Suz}} \newcommand{\Sy}{\mathrm{S}}  
\newcommand{\Tr}{\mathrm{Tr}}

\newtheorem{theorem}{Theorem}[section]
\newtheorem{lemma}[theorem]{Lemma}

\newtheorem{problem}[theorem]{Problem}

\theoremstyle{definition}

\newtheorem*{remark}{Remark}

\begin{document}

\title[Factorizations of almost simple orthogonal groups of plus type]{Factorizations of almost simple orthogonal groups of plus type}

\author[Li]{Cai Heng Li}
\address{(Li) SUSTech International Center for Mathematics and Department of Mathematics\\Southern University of Science and Technology\\Shenzhen 518055\\Guangdong\\P.~R.~China}
\email{lich@sustech.edu.cn}

\author[Wang]{Lei Wang}
\address{(Wang) School of Mathematics and Statistics\\Yunnan University\\Kunming 650091\\Yunnan\\P.~R.~China}
\email{wanglei@ynu.edu.cn}

\author[Xia]{Binzhou Xia}
\address{(Xia) School of Mathematics and Statistics\\The University of Melbourne\\Parkville 3010\\VIC\\Australia}
\email{binzhoux@unimelb.edu.au}

\begin{abstract}
This is the fifth one in a series of papers classifying the factorizations of almost simple groups with nonsolvable factors.
In this paper we deal with orthogonal groups of plus type.


\end{abstract}

\maketitle

\section{Introduction}

An expression $G=HK$ of a group $G$ as the product of subgroups $H$ and $K$ is called a \emph{factorization} of $G$, where $H$ and $K$ are called \emph{factors}. A group $G$ is said to be \emph{almost simple} if $S\leqslant G\leqslant\Aut(S)$ for some nonabelian simple group $S$, where $S=\Soc(G)$ is the \emph{socle} of $G$. In this paper, by a factorization of an almost simple group we mean that none its factors contains the socle. The main aim of this paper is to solve the long-standing open problem:

\begin{problem}\label{PrbXia1}
Classify factorizations of finite almost simple groups.
\end{problem}

Determining all factorizations of almost simple groups is a fundamental problem in the theory of simple groups, which was proposed by Wielandt~\cite[6(e)]{Wielandt1979} in 1979 at The Santa Cruz Conference on Finite Groups. It also has numerous applications to other branches of mathematics such as combinatorics and number theory, and has attracted considerable attention in the literature.
In what follows, all groups are assumed to be finite if there is no special instruction.

The factorizations of almost simple groups of exceptional Lie type were classified by Hering, Liebeck and Saxl~\cite{HLS1987}\footnote{In part~(b) of Theorem~2 in~\cite{HLS1987}, $A_0$ can also be $\G_2(2)$, $\SU_3(3)\times2$, $\SL_3(4).2$ or $\SL_3(4).2^2$ besides $\G_2(2)\times2$.} in 1987.
For the other families of almost simple groups, a landmark result was achieved by Liebeck, Praeger and Saxl~\cite{LPS1990} thirty years ago, which classifies the maximal factorizations of almost simple groups. (A factorization is said to be \emph{maximal} if both the factors are maximal subgroups.)
Then factorizations of alternating and symmetric groups are classified in~\cite{LPS1990}, and factorizations of sporadic almost simple groups are classified in~\cite{Giudici2006}.
This reduces Problem~\ref{PrbXia1} to the problem on classical groups of Lie type.
Recently, factorizations of almost simple groups with a factor having at least two nonsolvable composition factors are classified in~\cite{LX2019}\footnote{In Table~1 of~\cite{LX2019}, the triple $(L,H\cap L,K\cap L)=(\Sp_{6}(4),(\Sp_2(4)\times\Sp_{2}(16)).2,\G_2(4))$ is missing, and for the first two rows $R.2$ should be $R.P$ with $P\leqslant2$.}, and those with a factor being solvable are described in~\cite{LX} and~\cite{BL2021}.

As usual, for a finite group $G$, we denote by $G^{(\infty)}$ its solvable residual, that is, the smallest normal subgroup of $X$ such that $G/G^{(\infty)}$ is solvable.
For factorizations $G=HK$ with nonsolvable factors $H$ and $K$ such that $L=\Soc(G)$ is a classical group of Lie type, the triple $(L,H^{(\infty)},K^{(\infty)})$ is described in~\cite{LWX}. Based on this work, in the present paper we characterize the triples $(G,H,K)$ such that $G=HK$ with $H$ and $K$ nonsolvable, and $G$ is an orthogonal group of plus type.

For groups $H,K,X,Y$, we say that $(H,K)$ contains $(X,Y)$ if $H\geqslant X$ and $K\geqslant Y$, and that $(H,K)$ \emph{tightly contains} $(X,Y)$ if in addition $H^{(\infty)}=X^{(\infty)}$ and $K^{(\infty)}=Y^{(\infty)}$. 
Our main result is the following Theorem~\ref{ThmOmegaPlus}.
Note that it is elementary to determine the factorizations of $G/L$ as this group has relatively simple structure (and in particular is solvable).

\begin{theorem}\label{ThmOmegaPlus}
Let $G$ be an almost simple group with socle $L=\POm_{2m}^+(q)$, where $m\geqslant4$, and let $H$ and $K$ be subgroups of $G$ not containing $L$ such that both $H$ and $K$ have a unique nonsolvable composition factor. Then $G=HK$ if and only if (with $H$ and $K$ possibly interchanged) $G/L=(HL/L)(KL/L)$ and one of the following holds:
\begin{enumerate}[{\rm (a)}]
\item $(H,K)$ tightly contains $(X^\alpha,Y^\alpha)$ for some $(X,Y)$ in Table~$\ref{TabOmegaPlus}$ and $\alpha\in\Aut(L)$;
\item $(H^\alpha)^{(\infty)}=\lefthat(P.S)\leqslant\Pa_m[G]$ with $\alpha\in\Aut(L)$, $P<R$ and $S=\SL_a(q^b)$ ($m=ab$), $\Sp_a(q^b)'$ ($m=ab$), $\G_2(q^b)'$ ($m=6b$, $q$ even) or $\SL_2(13)$ ($m=6$, $q=3$), where $R=q^{m^2}$ is the unipotent radical of $\Pa_m[L]$ and $S$ is defined over $\bbF_{q^b}$ ($b$ is taken to be $1$ if $S=\SL_2(13)$), and $(K^\alpha)^{(\infty)}=\Omega_{2m-1}(q)\leqslant\N_1[G]$ such that $H^\alpha K^\alpha\supseteq R$;
\item $m=4$, $q$ is even, $(A^\tau)^{(\infty)}\leqslant K\leqslant A^\tau$ for some triality $\tau\in\Aut(L)$ and some subgroup $A$ of $G$ of the form $A=S\times\calO$ with $\Sp_6(q)\leqslant S\leqslant\GaSp_6(q)$ and $\calO\leqslant2$, and $H$ is a subgroup of $M:=\Pa_1[S]\times\calO$ such that $HK\supseteq\bfO_2(M)$ and $H\bfO_2(M)/\bfO_2(M)$ contains the field-extension subgroup $\Sp_2(q^2)$ of $M^{(\infty)}/\bfO_2(M^{(\infty)})=\Sp_4(q)$.
\end{enumerate}
\end{theorem}

\begin{remark}
Here are some remarks on Tables~\ref{TabOmegaPlus}:
\begin{enumerate}[{\rm(I)}]
\item The column $Z$ gives the smallest almost simple group with socle $L$ that contains $X$ and $Y$. In other words, $Z=\langle L,X,Y\rangle$.
It turns out that $Z=XY$ for all pairs $(X,Y)$.
\item The groups $X$, $Y$ and $Z$ are described in the corresponding lemmas whose labels are displayed in the last column. 
Information on $X\cap Y$ is also provided in these lemmas.
\item The description of groups $X$ and $Y$ are up to conjugations in $Z$ (see Lemma~\ref{LemXia04}(b) and Lemma~\ref{LemXia03}).
\end{enumerate}
\end{remark}

\begin{table}[htbp]
\captionsetup{justification=centering}
\caption{$(X,Y)$ for orthogonal groups of plus type}\label{TabOmegaPlus}
\begin{tabular}{|l|l|l|l|l|}
\hline
 & $Z$ & $X$ & $Y$ & Remark\\
\hline
1 & $\POm_{2m}^+(q)$ & $\lefthat(q^{m(m-1)/2}.\SL_a(q^b))$ ($m=ab$), & $\Omega_{2m-1}(q)$ & \ref{LemOmegaPlus02}--\ref{LemOmegaPlus08}\\
 & & $\lefthat(q^{m(m-1)/2}{:}\Sp_a(q^b)')$ ($m=ab$), & & \\
 & & $\lefthat(q^{m(m-1)/2}{:}\G_2(q^b)')$ ($m=6b$, $q$ even), & & \\
 & & $\lefthat\SL_m(q)$, $\lefthat\SU_m(q)$ ($m$ even), $\PSp_m(q)$ & & \\
\hline
2 & $\Omega_{2m}^+(2)$ & $\Omega_m^+(4).2$, $\SL_{m/2}(4).2$, $\Sp_{m/2}(4).2$, & $\Sp_{2m-2}(2)$ & \ref{LemOmegaPlus11}, \ref{LemOmegaPlus10}--\ref{LemOmegaPlus45}\\
 & & $\Sp_{m/2}(4).4$, $\SU_{m/2}(4).4$ ($m/2$ even) & & \\
3 & $\Omega_{2m}^+(4).2$ & $\Omega_m^+(16).4$, $\SL_{m/2}(16).4$, $\Sp_{m/2}(16).4$, & $\GaSp_{2m-2}(4)$ & \ref{LemOmegaPlus11}, \ref{LemOmegaPlus10}--\ref{LemOmegaPlus45}\\
 & & $\Sp_{m/2}(16).8$, $\SU_{m/2}(16).8$ ($m/2$ even) & & \\
\hline
4 & $\POm_8^+(q)$ & $\Omega_8^-(q^{1/2})$, $\Omega_7(q)$, $\Omega_6^+(q)$, $\Omega_6^-(q)$, $\Omega_5(q)$, & $\Omega_7(q)$ & \ref{LemOmegaPlus26}, \ref{LemOmegaPlus27}, \ref{LemOmegaPlus42}, \ref{LemOmegaPlus28}\\
 & & $q^5{:}\Omega_5(q)$ ($q$ odd), $q^4{:}\Omega_4^-(q)$ ($q$ odd) & & \\
5 & $\Omega_{12}^+(q)$ & $\G_2(q)$ & $\Sp_{10}(q)$ & $q$ even; \ref{LemOmegaPlus06} \\
6 & $\POm_{16}^+(q)$ & $\Omega_9(q)$ & $\Omega_{15}(q)$ & \ref{LemOmegaPlus29} \\
\hline
7 & $\POm_{2m}^+(q)$ & $\lefthat(q^{m(m-1)/2}{:}\SL_m(q))$ & $\Omega_{2m-2}^-(q)$ & \ref{LemOmegaPlus13}\\
8 & $\Omega_{2m}^+(2).c$ & $\SL_m(2).2$ & $\Omega_{2m-2}^-(2).c$ & $c=(2,m-1)$; \ref{LemOmegaPlus14}\\
9 & $\Omega_{2m}^+(2)$ & $\SL_m(2)$ & $\Omega_{2m-2}^-(2).2$ & \ref{LemOmegaPlus15}\\
10 & $\Omega_{2m}^+(4).2$ & $\SL_m(4).2$ & $\Omega_{2m-2}^-(4).4$ & \ref{LemOmegaPlus16}\\
\hline
11 & $\POm_{2m}^+(q)$ & $q^{2m-2}{:}\Omega_{2m-2}^+(q)$ & $\lefthat\SU_m(q)$ & $m$ even; \ref{LemOmegaPlus17}\\
12 & $\Omega_{2m}^+(2)$ & $\Omega_{2m-2}^+(2)$ & $\SU_m(2).2$ & $m$ even; \ref{LemOmegaPlus19}\\
13 & $\Omega_{2m}^+(2)$ & $\Omega_{2m-2}^+(2).2$ & $\SU_m(2)$ & $m$ even; \ref{LemOmegaPlus18}\\
14 & $\Omega_{2m}^+(4).2$ & $\Omega_{2m-2}^+(4).2$ & $\SU_m(4).4$ & $m$ even; \ref{LemOmegaPlus19}, \ref{LemOmegaPlus22}\\
\hline
15 & $\Omega_8^+(2)$ & $\Sy_5$, $\A_5{:}4$, $2^4{:}\A_5$, $\A_6$, $2^5{:}\A_6$, & $\Sp_6(2)$ & \ref{LemOmegaPlus30} \\
 & & $\A_7$, $2^6{:}\A_7$, $\A_8$, $\A_9$ & & \\
16 & $\Omega_8^+(2)$ & $2^6{:}\A_7$, $\A_8$, $\A_9$ & $\SU_4(2)$ & \ref{LemOmegaPlus30} \\
17 & $\Omega_8^+(2)$ & $\A_8$ & $\SU_4(2).2$ & \ref{LemOmegaPlus30} \\
18 & $\Omega_8^+(2)$ & $2^4{:}\A_5$, $2^5{:}\A_6$, $2^6{:}\A_7$, $\A_8$, $2^6{:}\A_8$ & $\A_9$ & \ref{LemOmegaPlus30} \\
\hline
19 & $\POm_8^+(3)$ & $3^4{:}\Sy_5$, $3^4{:}(4\times\A_5)$, $(3^5{:}2^4){:}\A_5$, & $\Omega_7(3)$ & \ref{LemOmegaPlus31} \\
 & & $\A_9$, $\SU_4(2)$, $\Sp_6(2)$, $\Omega_8^+(2)$ & & \\
20 & $\POm_8^+(3)$ & $2^6{:}\A_7$, $\A_8$, $2^6{:}\A_8$, $\A_9$, & $3^6{:}\PSL_4(3)$ & \ref{LemOmegaPlus31} \\
 &  & $2.\PSL_3(4)$, $\Sp_6(2)$, $\Omega_6^-(3)$ & & \\
21 & $\POm_8^+(3)$ & $3^6{:}\SL_3(3)$, $3^{3+3}{:}\SL_3(3)$, & $\Omega_8^+(2)$ & \ref{LemOmegaPlus31} \\
 & & $3^{6+3}{:}\SL_3(3)$, $3^6{:}\PSL_4(3)$ & & \\
\hline
22 & $\Omega_8^+(4).2$ & $(\SL_2(16)\times5).4$, $\GaL_2(16)\times3$, & $\GaSp_6(4)$ & \ref{LemOmegaPlus40}, \ref{LemOmegaPlus25}\\
 & & $\GaL_2(16)$, $\SL_2(16).4$ & & \\
23 & $\Omega_8^+(4).2$ & $\Omega_8^-(2).2$ & $\SU_4(4).4$ & \ref{LemOmegaPlus23}\\
24 & $\Omega_8^+(16).4$ & $\Omega_8^-(4).4$ & $\SU_4(16).8$ & \ref{LemOmegaPlus23}\\
25 & $\Omega_{12}^+(2)$ & $3.\PSU_4(3)$, $3.\M_{22}$ & $\Sp_{10}(2)$ & \ref{LemOmegaPlus34}\\
26 & $\POm_{12}^+(3)$ & $3^{15}.\PSL_2(13)$ & $\Omega_{11}(3)$ & \ref{LemOmegaPlus02}\\
27 & $\Omega_{16}^+(2)$ & $\Omega_8^-(2).2$, $\Sp_6(4).2$, $\Omega_6^+(4).2$, & $\Sp_{14}(2)$ & \ref{LemOmegaPlus35}\\
 & & $\Omega_6^-(4).2$, $\Sp_4(4).2$ & & \\
28 & $\Omega_{16}^+(4).2$ & $\Omega_8^-(4).4$, $\Sp_6(16).4$, $\Omega_6^+(16).4$, & $\GaSp_{14}(4)$ & \ref{LemOmegaPlus35}\\
 & & $\Omega_6^-(16).4$, $\Sp_4(16).4$ & & \\
29 & $\Omega_{24}^+(2)$ & $\G_2(4).2$, $\G_2(4).4$, $3.\Suz$, $\Co_1$ & $\Sp_{22}(2)$ & \ref{LemOmegaPlus36}, \ref{LemOmegaPlus38}\\
30 & $\Omega_{24}^+(4).2$ & $\G_2(16).4$, $\G_2(16).8$ & $\GaSp_{22}(4)$ & \ref{LemOmegaPlus36}\\
31 & $\Omega_{32}^+(2)$ & $\Sp_8(4).2$ & $\Sp_{30}(2)$ & \ref{LemOmegaPlus09}\\
32 & $\Omega_{32}^+(4).2$ & $\Sp_8(16).4$ & $\GaSp_{30}(4)$ & \ref{LemOmegaPlus09}\\
\hline
\end{tabular}
\vspace{3mm}
\end{table}

\section{Preliminaries}

In this section we collect some elementary facts regarding group factorizations.

\begin{lemma}\label{LemXia01}
Let $G$ be a group, let $H$ and $K$ be subgroups of $G$, and let $N$ be a normal subgroup of $G$. Then $G=HK$ if and only if $HK\supseteq N$ and $G/N=(HN/N)(KN/N)$.
\end{lemma}

\begin{proof}
If $G=HK$, then $HK\supseteq N$, and taking the quotient modulo $N$ we obtain
\[
G/N=(HN/N)(KN/N).
\]
Conversely, suppose that $HK\supseteq N$ and $G/N=(HN/N)(KN/N)$. Then 
\[
G=(HN)(KN)=HNK 
\]
as $N$ is normal in $G$. Since $N\subseteq HK$, it follows that $G=HNK\subseteq H(HK)K=HK$, which implies $G=HK$.
\end{proof}

Let $L$ be a nonabelian simple group. We say that $(H,K)$ is a \emph{factor pair} of $L$ if $H$ and $K$ are subgroups of $\Aut(L)$ such that $HK\supseteq L$. For an almost simple group $G$ with socle $L$ and subgroups $H$ and $K$ of $G$, Lemma~\ref{LemXia01} shows that $G=HK$ if and only if $G/L=(HL/L)(KL/L)$ and $(H,K)$ is a factor pair.
As the group $G/L$ has a simple structure (and in particular is solvable), it is elementary to determine the factorizations of $G/L$.
Thus to know all the factorizations of $G$ is to know all the factor pairs of $L$.
Note that, if $(H,K)$ is a factor pair of $L$, then any pair of subgroups of $\Aut(L)$ containing $(H,K)$ is also a factor pair of $L$.
Hence we have the following:

\begin{lemma}\label{LemXia02}
Let $G$ be an almost simple group with socle $L$, and let $H$ and $K$ be subgroups of $G$ such that $(H,K)$ contains some factor pair of $L$. Then $G=HK$ if and only if $G/L=(HL/L)(KL/L)$.
\end{lemma}

In light of Lemma~\ref{LemXia02}, the key to determine the factorizations of $G$ with nonsolvable factors is to determine the minimal ones (with respect to the containment) among factor pairs of $L$ with nonsolvable subgroups.

\begin{lemma}\label{LemXia03}
Let $L$ be a nonabelian simple group, and let $(H,K)$ be a factor pair of $L$.
Then $(H^\alpha,K^\alpha)$ and $(H^x,K^y)$ are factor pairs of $L$ for all $\alpha\in\Aut(L)$ and $x,y\in L$.
\end{lemma}

\begin{proof}
It is evident that $H^\alpha K^\alpha=(HK)^\alpha\supseteq L^\alpha=L$. Hence $(H^\alpha,K^\alpha)$ is a factor pair.
Since $xy^{-1}\in L\subseteq HK$, there exist $h\in H$ and $k\in K$ such that $xy^{-1}=hk$. Therefore, 
\[
H^xK^y=x^{-1}Hxy^{-1}Ky=x^{-1}HhkKy=x^{-1}HKy\supseteq x^{-1}Ly=L,
\]
which means that $(H^x,K^y)$ is a factor pair.
\end{proof}

The next lemma is~\cite[Lemma~2(i)]{LPS1996}.

\begin{lemma}\label{LemXia05}
Let $G$ be an almost simple group with socle $L$, and let $H$ and $K$ be subgroups of $G$ not containing $L$. If $G=HK$, then $HL\cap KL=(H\cap KL)(K\cap HL)$. 
\end{lemma}

The following lemma implies that we may consider specific representatives of a conjugacy class of subgroups when studying factorizations of a group.

\begin{lemma}\label{LemXia04}
Let $G=HK$ be a factorization. Then for all $x,y\in G$ we have $G=H^xK^y$ with $H^x\cap K^y\cong H\cap K$.
\end{lemma}
  
\begin{proof}
As $xy^{-1}\in G=HK$, there exists $h\in H$ and $k\in K$ such that $xy^{-1}=hk$. Thus
\[
H^xK^y=x^{-1}Hxy^{-1}Ky=x^{-1}HhkKy=x^{-1}HKy=x^{-1}Gy=G,
\]
and
\[
H^x\cap K^y=(H^{xy^{-1}}\cap K)^y\cong H^{xy^{-1}}\cap K=H^{hk}\cap K=H^k\cap K=(H\cap K)^k\cong H\cap K.\qedhere
\]
\end{proof}

\section{Notation}\label{SecOmegaPlus01}

We follow the convention as in~\cite{LPS1990} to denote $\SL_n^+(q)=\SL_n(q)$ and $\SL_n^-(q)=\SU_n(q)$. 
For an almost simple group $G$, a $\max^-$ subgroup of $G$ refers to a maximal one among the core-free subgroups of $G$.

Throughout this paper, let $q=p^f$ be a power of a prime $p$, let $m\geqslant4$ be an integer, let $\,\overline{\phantom{\varphi}}\,$ be the homomorphism from $\GaO_{2m}^+(q)$ to $\mathrm{P\Gamma O}_{2m}^+(q)$ modulo scalars, let $\tau$ be a triality automorphism of $\POm_8^+(q)$, let $V$ be a vector space of dimension $2m$ over $\bbF_q$ equipped with a nondegenerate quadratic form $Q$ of plus type, whose associated bilinear form is $\beta$, let $\perp$ denote the perpendicularity with respect to $\beta$, let $e_1,f_1,\dots,e_m,f_m$ be a standard basis for $V$ as in~\cite[2.2.3]{LPS1990}, let $\phi\in\GaO(V)$ such that
\[
\phi\colon a_1e_1+b_1f_1+\dots+a_me_m+b_mf_m\mapsto a_1^pe_1+b_1^pf_1+\dots+a_m^pe_m+b_m^pf_m
\]
for $a_1,b_1\dots,a_m,b_m\in\bbF_q$, let $\gamma$ be the involution in $\GO(V)$ swapping $e_i$ and $f_i$ for all $i\in\{1,\dots,m\}$, let $U=\langle e_1,\dots,e_m\rangle_{\bbF_q}$, let $U_1=\langle e_2,\dots,e_m\rangle_{\bbF_q}$, let $U_2=\langle e_3,\dots,e_m\rangle_{\bbF_q}$, and let $W=\langle f_1,\dots,f_m\rangle_{\bbF_q}$.
From~\cite[3.7.4~and~3.8.2]{Wilson2009} we see that $\Omega(V)_U=\Pa_m[\Omega(V)]$ has a subgroup $R{:}T$, where
\[
R=q^{m(m-1)/2}
\]
is the kernel of $\Omega(V)_U$ acting on $U$, and
\[
T=\SL_m(q)
\]
stabilizes both $U$ and $W$. The action of $T$ on $U$ determines that on $W$ in the way described in~\cite[Lemma~2.2.17]{BG2016}, from which we can see that $\gamma$ normalizes $T$ and induces the transpose inverse automorphism of $T$.
For $w\in V$ with $Q(w)\neq0$, let $r_w$ be the reflection in $w$ defined by
\[
r_w\colon V\to V,\quad v\mapsto v-\frac{\beta(v,w)}{Q(w)}w.
\]
Take $\mu\in\bbF_q$ such that the polynomial $x^2+x+\mu$ is irreducible over $\bbF_q$ (note that this implies $\mu\neq0$), and let
\begin{align*}
&u=e_1+e_2+\mu f_2,\\
&u'=(1-\mu^2)e_1+(\mu^2-1+\mu^{-1})e_2+\mu^2f_1+\mu^2f_2.
\end{align*}
Then $Q(e_1+f_1)=\beta(e_1+f_1,u)=\beta(e_1+f_1,u')=1$ and $Q(u)=Q(u')=\mu$. Hence $\langle e_1+f_1,u\rangle_{\bbF_q}$ and $\langle e_1+f_1,u'\rangle_{\bbF_q}$ are both nondegenerate $2$-subspaces of minus type in $V$, and there exists $y\in\Omega(V)$ (respectively, $y\in\mathrm{O}(V)\setminus\Omega(V)$) such that
\[
(e_1+f_1)^y=e_1+f_1\ \text{ and }\ u^y=u'.
\]

If $m=2\ell$ is even, then let $V_\sharp$ be a vector space of dimension $m$ over $\bbF_{q^2}$ with the same underlying set as $V$, let $\Tr$ be the trace of the field extension $\bbF_{q^2}/\bbF_q$, and let $\lambda\in\bbF_{q^2}$ such that
\[
\Tr(\lambda)=\lambda+\lambda^q=1.
\]
In this case, we fix the notation for some field-extension subgroups of $\mathrm{O}(V)$ as follows.

Firstly, we may equip $V_\sharp$ with a nondegenerate unitary form $\beta_\sharp$ such that $Q(v)=\beta_\sharp(v,v)$ for all $v\in V$, and thus $\GU(V_\sharp)<\mathrm{O}(V)$. Take a standard $\bbF_{q^2}$-basis $E_1,F_1,\dots,E_\ell,F_\ell$ for $V_\sharp$, so that
\[
\beta_\sharp(E_i,E_j)=\beta_\sharp(F_i,F_j)=0,\quad\beta_\sharp(E_i,F_j)=\delta_{i,j}
\]
for all $i,j\in\{1,\dots,\ell\}$. Then we have
\begin{align}
&Q(\lambda E_1)=\beta_\sharp(\lambda E_1,\lambda E_1)=0,\label{EqnOmegaPlus01}\\
&Q(F_1)=\beta_\sharp(F_1,F_1)=0,\label{EqnOmegaPlus02}\\
&Q(\lambda E_1+F_1)=\beta_\sharp(\lambda E_1+F_1,\lambda E_1+F_1)=\lambda+\lambda^q=1,\nonumber
\end{align}
and hence
\begin{equation}\label{EqnOmegaPlus03}
\beta(\lambda E_1,F_1)=Q(\lambda E_1+F_1)-Q(\lambda E_1)-Q(F_1)=1.
\end{equation}
From~\eqref{EqnOmegaPlus01}--\eqref{EqnOmegaPlus03} we see that $(\lambda E_1,F_1)$ is a hyperbolic pair with respect to $Q$.
Thus we may assume without loss of generality that
\[
e_1=\lambda E_1\ \text{ and }\ f_1=F_1.
\]
Then $Q(e_1+f_1)=\beta_\sharp(e_1+f_1,e_1+f_1)=1$, which shows that $e_1+f_1$ is a nonsingular vector in both $V$ (with respect to $Q$) and $V_\sharp$ (with respect to $\beta_\sharp$). Moreover,~\eqref{EqnOmegaPlus01} shows that $e_1$ is a singular vector in both $V$ and $V_\sharp$.
Let $\xi\in\GaU(V_\sharp)$ such that
\[
\xi\colon a_1E_1+b_1F_1+\dots+a_\ell E_\ell+b_\ell F_\ell\mapsto a_1^pE_1+b_1^pF_1+\dots+a_\ell^pE_\ell+b_\ell^pF_\ell
\]
for $a_1,b_1\dots,a_\ell,b_\ell\in\bbF_{q^2}$.

Secondly, we may equip $V_\sharp$ with a nondegenerate quadratic form $Q_\sharp$ of plus type such that $Q(v)=\Tr(Q_\sharp(v))$ for all $v\in V$, and thus $\mathrm{O}(V_\sharp)<\mathrm{O}(V)$. For $w\in V_\sharp$ with $Q_\sharp(w)\neq0$, let $r'_w$ be the reflection in $w$ defined by
\[
r'_w\colon V_\sharp\to V_\sharp,\quad v\mapsto v-\frac{Q_\sharp(v+w)-Q_\sharp(v)-Q_\sharp(w)}{Q_\sharp(w)}w.
\]
Take a standard $\bbF_{q^2}$-basis $E_1,F_1,\dots,E_\ell,F_\ell$ for $V_\sharp$. Then we have
\begin{align}
&Q(\lambda E_1)=\Tr(Q_\sharp(\lambda E_1))=\Tr(0)=0,\label{EqnOmegaPlus04}\\
&Q(F_1)=\Tr(Q_\sharp(F_1))=\Tr(0)=0,\label{EqnOmegaPlus05}\\
&Q(\lambda E_1+F_1)=\Tr(Q_\sharp(\lambda E_1+F_1))=\Tr(\lambda)=\lambda+\lambda^q=1,\nonumber
\end{align}
and hence
\begin{equation}\label{EqnOmegaPlus06}
\beta(\lambda E_1,F_1)=Q(\lambda E_1+F_1)-Q(\lambda E_1)-Q(F_1)=1.
\end{equation}
From~\eqref{EqnOmegaPlus04}--\eqref{EqnOmegaPlus06} we see that $(\lambda E_1,F_1)$ is a hyperbolic pair with respect to $Q$.
Thus we may assume without loss of generality that
\[
e_1=\lambda E_1\ \text{ and }\ f_1=F_1.
\]
It follows that $Q(e_1+f_1)=1$ and $Q_\sharp(e_1+f_1)=\lambda$, whence $e_1+f_1$ is a nonsingular vector in both $V$ (with respect to $Q$) and $V_\sharp$ (with respect to $Q_\sharp$). Let $\psi\in\GaO(V_\sharp)$ such that
\[
\psi\colon a_1E_1+b_1F_1+\dots+a_\ell E_\ell+b_\ell F_\ell\mapsto a_1^pE_1+b_1^pF_1+\dots+a_\ell^pE_\ell+b_\ell^pF_\ell
\]
for $a_1,b_1\dots,a_\ell,b_\ell\in\bbF_{q^2}$.

\section{Factorizations with socle $\Sp_6(q)$ and a factor normalizing $\mathrm{G}_2(q)$}

For applications in the next section, we need to know the factorizations of almost simple groups with socle $\Sp_6(q)$ for even $q$ such that a factor normalizes $\mathrm{G}_2(q)$. This will be classified in this section (Lemma~\ref{LemSymplectic12}).

\begin{lemma}\label{LemSymplectic07}
Let $Z=\Sp_6(q)$ with $q$ even, and let $Y=\G_2(q)<Z$. 
\begin{enumerate}[{\rm (a)}]
\item If $X=\Omega_6^\varepsilon(q)<Z$ with $\varepsilon\in\{+,-\}$, then $Z=XY$ with $X\cap Y=\SL_3^\varepsilon(q)$.
\item If $X=\Omega_5(q)<\Omega_6^\varepsilon(q)<Z$ with $\varepsilon\in\{+,-\}$, then $Z=XY$ with $X\cap Y=\SL_2(q)$.
\end{enumerate}
\end{lemma}

\begin{proof}
First assume that $X$ is defined in part~(a). Let $M=\mathrm{O}_6^\varepsilon(q)$ be a subgroup of $Z$ containing $X$. It is shown in~\cite[5.2.3(b)]{LPS1990} that $M\cap Y=\SL_3^\varepsilon(q).2$.
Thus $X\cap Y=\SL_3^\varepsilon(q)$ or $\SL_3^\varepsilon(q).2$ as $X$ has index $2$ in $M$. However, as $q$ is even, we see from~\cite[Tables~8.8--8.11]{BHR2013} that $X$ has no subgroup isomorphic to $\SL_3^\varepsilon(q).2$. Hence $X\cap Y=\SL_3^\varepsilon(q)$, and it follows that
\[
\frac{|Y|}{|X\cap Y|}=\frac{|\G_2(q)|}{|\SL_3^\varepsilon(q)|}=q^3(q^3+\varepsilon1)=\frac{|\Sp_6(q)|}{|\Omega_6^\varepsilon(q)|}=\frac{|Z|}{|X|}.
\]
Consequently, $Z=XY$.

Next assume that $X$ is defined in part~(b). Let $M=\Omega_6^\varepsilon(q)<Z$ such that $X<M$. Part~(a) asserts that $Z=MY$ with $M\cap Y=\SL_3^\varepsilon(q)$. By~\cite[Lemma~4.5]{LWX-Linear} and~\cite[Lemma~4.7]{LWX-Unitary} we have $M=X(M\cap Y)$ with $X\cap(M\cap Y)=X\cap\SL_3^\varepsilon(q)=\SL_2(q)$. Hence $Z=XY$ with $X\cap Y=\SL_2(q)$.
\end{proof}

\begin{remark}
If we let $Y=\G_2(q)'$ in Lemma~\ref{LemSymplectic07}, then computation in \magma~\cite{BCP1997} shows that the conclusion $Z=XY$ would not hold for $(q,\varepsilon)=(2,+)$.
\end{remark}

Note that, if $\Sp_6(q)\leqslant G\leqslant\GaSp_6(q)$ with even $q$ and $A=\Pa_1[G]$, then the largest normal $2$-subgroup of $A$ is $\bfO_2(A)=q^5$ and the quotient $A^{(\infty)}/\bfO_2(A)$ is the symplectic group $\Sp_4(q)$.

\begin{lemma}\label{LemSymplectic09}
Let $L=\Sp_6(q)$ with $q\geqslant4$ even, let $G$ be an almost simple group with socle $L$, let $A=\Pa_1[G]$, let $H$ be a subgroup of $A$, and let $K$ be a subgroup of $G$ with $K^{(\infty)}=\G_2(q)$. Then $G=HK$ if and only if $G/L=(HL/L)(KL/L)$, $HK\supseteq\bfO_2(A)$ and $H\bfO_2(A)/\bfO_2(A)$ contains the field-extension subgroup $\Sp_2(q^2)$ of $A^{(\infty)}/\bfO_2(A)$.
\end{lemma}

\begin{proof}
By~\cite[Lemma~4.3]{LWX-Linear} we have $A^{(\infty)}K^{(\infty)}\supseteq L$ with $A^{(\infty)}\cap K^{(\infty)}=q^{2+3}{:}\Sp_2(q)$. Moreover, we see from~\cite[4.3.7]{Wilson2009} that $\bfO_2(A)\cap K^{(\infty)}=q^2$, that is, $(A^{(\infty)}\cap K^{(\infty)})\cap\bfO_2(A)=q^2$. Hence
\[
(A^{(\infty)}\cap K^{(\infty)})\bfO_2(A)/\bfO_2(A)=q^3{:}\Sp_2(q)
\]
is the stabilizer of a vector in $A^{(\infty)}/\bfO_2(A)=\Sp_4(q)$. Since $A^{(\infty)}\cap K^{(\infty)}$ is normal in $A\cap K$, it follows that $(A\cap K)\bfO_2(A)/\bfO_2(A)$ has a normal subgroup $q^3{:}\Sp_2(q)$.

First suppose that $G=HK$. Then $G/L=(HL/L)(KL/L)$ and $HK\supseteq\bfO_2(A)$. Since $H\leqslant A$, we obtain $A=H(A\cap K)$ and thus
\[
A/\bfO_2(A)=\big(H\bfO_2(A)/\bfO_2(A)\big)\big((A\cap K)\bfO_2(A)/\bfO_2(A)\big).
\]
This implies that $H\bfO_2(A)/\bfO_2(A)$ acts transitively on the set of $1$-spaces in $\bbF_q^4$. By a result of Hering (see~\cite[Lemma~3.1]{LPS2010}), it then follows that $H\bfO_2(A)/\bfO_2(A)$ contains the field-extension subgroup $\Sp_2(q^2)$ of $A^{(\infty)}/\bfO_2(A)$.

Now suppose that $G/L=(HL/L)(KL/L)$, $HK\supseteq\bfO_2(A)$ and $H\bfO_2(A)/\bfO_2(A)$ contains the field-extension subgroup $\Sp_2(q^2)$ of $A^{(\infty)}/\bfO_2(A)$. By~\cite[Lemma~4.2]{LWX-Linear}, $H\bfO_2(A)/\bfO_2(A)\geqslant\Sp_2(q^2)$ is transitive on the set of vectors in $\bbF_q^4$. Hence
\begin{align*}
A^{(\infty)}/\bfO_2(A)=\Sp_4(q)&\subseteq\big(H\bfO_2(A)/\bfO_2(A)\big)\big((A^{(\infty)}\cap K^{(\infty)})\bfO_2(A)/\bfO_2(A)\big)\\
&=\big(H(A^{(\infty)}\cap K^{(\infty)})\big)\bfO_2(A)/\bfO_2(A).
\end{align*}
This together with the condition $HK\supseteq\bfO_2(A)$ implies that $HK\supseteq A^{(\infty)}$. Hence $HK\supseteq A^{(\infty)}K^{(\infty)}=L$, and so by Lemma~\ref{LemXia01} we obtain $G=HK$ as $G/L=(HL/L)(KL/L)$.
\end{proof}

The next lemma is verified by computation in \magma~\cite{BCP1997}.

\begin{lemma}\label{LemSymplectic10}
Let $Z=\GaSp_6(4)$, and let $Y=\GaG_2(4)<Z$.
\begin{enumerate}[{\rm (a)}]
\item If $X=(\SL_2(16)\times5).4<\GaO_6^+(4)<Z$ (there are exactly two conjugacy classes of such subgroups $X$), then $Z=XY$ with $X\cap Y=5$.
\item If $X=\GaL_2(16)\times3<\GaO_6^-(4)<Z$ (there is a unique conjugacy class of such subgroups $X$), then $Z=XY$ with $X\cap Y=3$.
\end{enumerate}
\end{lemma}

We can now characterize factorizations of $6$-dimensional symplectic groups with a factor normalizing $\G_2(q)$.

\begin{lemma}\label{LemSymplectic12}
Let $L=\Sp_6(q)$ with $q\geqslant4$ even, let $G$ be an almost simple group with socle $L$, let $A=\Pa_1[G]$, let $H$ be a subgroup of $G$ not containing $L$ such that $H$ has a unique nonsolvable composition factor, and let $K$ be a subgroup of $G$ with $K^{(\infty)}=\G_2(q)$. Then $G=HK$ if and only if $G/L=(HL/L)(KL/L)$ and one of the following holds:
\begin{enumerate}[{\rm (a)}]
\item $H$ is a subgroup of $A$ such that $HK\supseteq\bfO_2(A)$ and $H\bfO_2(A)/\bfO_2(A)$ contains the field-extension subgroup $\Sp_2(q^2)$ of $A^{(\infty)}/\bfO_2(A)=\Sp_4(q)$;
\item $(H,K)$ tightly contains $(\Omega_6^+(q),\G_2(q))$, $(\Omega_6^-(q),\G_2(q))$ or $(\Sp_4(q),\G_2(q))$;
\item $G=\GaSp_6(4)$, and $(H,K)$ tightly contains some $(X,Y)$ as in Lemma~$\ref{LemSymplectic10}$.
\end{enumerate}
\end{lemma}

\begin{proof}
By Lemmas~\ref{LemXia01},~\cite[Lemma~4.10]{LWX-Linear} and Lemmas~\ref{LemSymplectic07}--\ref{LemSymplectic10}, it suffices to prove that if $G=HK$ then either $H\leqslant\Pa_1[G]$ or one of~(b) or~(c) of the lemma holds. For $q=4$, this is directly verified by computation in \magma~\cite{BCP1997}.
Thus we assume $q\geqslant8$ in what follows. Suppose $G=HK$ with $H\nleqslant\Pa_1[G]$. Let $M$ be a maximal subgroup of $G$ containing $H$.
Then $G=MK$ with $M\neq\Pa_1[G]$, and so we derive from~\cite[Theorem~A]{LPS1990} and~\cite{LPS1996} that $M\cap L=\mathrm{O}_6^+(q)$, $\mathrm{O}_6^-(q)$ or $\N_2[L]$.

First assume that $M\cap L=\mathrm{O}_6^\varepsilon(q)$ with $\varepsilon\in\{+,-\}$. Then $M$ is an almost simple group with socle $\Omega_6^\varepsilon(q)$ and $M=H(M\cap K)$. By Lemma~\ref{LemSymplectic07}, the group $M\cap K$ has a normal subgroup $\SL_3^\varepsilon(q)$. Note the isomorphism $\Omega_6^\varepsilon(q)\cong\SL_4^\varepsilon(q)$ from the Klein correspondence.
Then by~\cite[Theorem~A]{LPS1990} and~\cite{LPS1996}, we conclude from the factorization $M=H(M\cap K)$ that either $\varepsilon=+$ and $H\leqslant\N_1[M]$, or $\varepsilon=-$ and $H\leqslant\N_1[M]$ or $\Pa_1[M]$.
However, this yields that $H\leqslant\Pa_1[G]$, a contradiction.

Next assume that $M\cap L=\N_2[L]=\Sp_4(q)\times\Sp_2(q)$. As $H$ has a unique nonsolvable composition factor,~\cite[Theorem~5.1]{LWX} shows that either $H^{(\infty)}\in\{\Sp_4(q),\Omega_6^+(q),\Omega_6^-(q)\}$ or the unique nonsolvable composition factor of $H$ is $\Sp_2(q^2)$. Obviously, $H^{(\infty)}$ is not isomorphic to $\Omega_6^+(q)$ or $\Omega_6^-(q)$ as $H\leqslant M$. If $H$ has $\Sp_2(q^2)$ as a composition factor, then $H\leqslant N$ for some subgroup $N$ of $M$ such that $N\cap L=(\Sp_2(q^2).2)\times\Sp_2(q)$. However, $G$ does not have such a factorization $G=NK$ by~\cite[Theorem~1.1]{LX2019}. Hence $H^{(\infty)}=\Sp_4(q)$, and so $(H,K)$ tightly contains $(H^{(\infty)},K^{(\infty)})=(\Sp_4(q),\G_2(q))$, as in part~(b) of the lemma.
\end{proof}

\section{Factorizations with socle $\POm_8^+(q)$ and a factor normalizing $\Omega_7(q)$}\label{SecOmegaPlus02}

As a special case of Theorem~\ref{ThmOmegaPlus}, the factorizations $G=HK$ of almost simple groups $G$ with socle $\POm_8^+(q)$ will be classified in this section for those with $H\leqslant A$ and $(A^\tau)^{(\infty)}\leqslant K\leqslant A^\tau$ for certain subgroups $A$ of $G$ (see Lemma~\ref{LemOmegaPlus42}). By~\cite[Propositions~4.1.6~and~4.1.7]{KL1990} we have $\N_1[\POm_8^+(q)]=\Omega_7(q)$. 

\begin{lemma}\label{LemOmegaPlus26}
Let $Z=\POm_8^+(q)$ with $m=4$, let $Y=\N_1[Z]=\Omega_7(q)$, and let $X=Y^\tau$. Then $Z=XY$ with $X\cong Y=\Omega_7(q)$ and $X\cap Y=\G_2(q)$.
\end{lemma}

\begin{proof}
According to~\cite[Proposition~3.1.1(iv)]{Kleidman1987}, $X\cap Y=\G_2(q)$. Therefore,
\[
\frac{|X|}{|X\cap Y|}=\frac{|\Omega_7(q)|}{|\G_2(q)|}=\frac{q^3(q^4-1)}{(2,q-1)}=\frac{(2,q-1)q^3(q^4-1)}{(4,q^4-1)}=\frac{|\POm_8^+(q)|}{|\Omega_7(q)|}=\frac{|Z|}{|Y|},
\]
and so $Z=XY$.
\end{proof}

\begin{lemma}\label{LemOmegaPlus27}
Let $Z=\POm_8^+(q)$ with $m=4$, let $Y=\N_1[Z]=\Omega_7(q)$, let $M=Y^\tau$, and let $X$ be a subgroup of $M$ isomorphic to $\Omega_6^+(q)$, $\Omega_6^-(q)$ or $\Omega_5(q)$. Then $Z=XY$ with $X=\Omega_7(q)$ and
\[
X\cap Y=
\begin{cases}
\SL_3(q)&\textup{if }X\cong\Omega_6^+(q)\\
\SU_3(q)&\textup{if }X\cong\Omega_6^-(q)\\
\SL_2(q)&\textup{if }X\cong\Omega_5(q).
\end{cases}
\]
\end{lemma}

\begin{proof}
It follows from Lemma~\ref{LemOmegaPlus26} that $Z=MY$ with $M\cap Y=\G_2(q)$. Then by~\cite[Lemmas~4.3,~4.4,~5.1]{LWX-Omega} and Lemma~\ref{LemSymplectic07}, we have $Z=MY=X(M\cap Y)Y=XY$ with
\[
X\cap Y=X\cap(M\cap Y)=X\cap\G_2(q)=\SL_2(q)=
\begin{cases}
\SL_3(q)&\textup{if }X\cong\Omega_6^+(q)\\
\SU_3(q)&\textup{if }X\cong\Omega_6^-(q)\\
\SL_2(q)&\textup{if }X\cong\Omega_5(q).
\end{cases}
\qedhere
\]
\end{proof}

\begin{lemma}\label{LemOmegaPlus40}
Let $Z=\Omega(V){:}\langle\phi\rangle=\Omega_8^+(4){:}2$ with $m=q=4$, let $Y=\GaSp_6(4)<Z$ (there are precisely three conjugacy classes of such subgroups $Y$).
\begin{enumerate}[{\rm (a)}]
\item If $X=(\SL_2(16)\times5).4<\GaO_6^+(4)<Y^\tau$, then $Z=XY$ with $X\cap Y=5$.
\item if $X=\GaL_2(16)\times3<\GaO_6^-(4)<Y^\tau$, then $Z=XY$ with $X\cap Y=3$.
\end{enumerate}
\end{lemma}

\begin{proof}
By Lemma~\ref{LemOmegaPlus26} we have $Z'=(Y^\tau)'Y'$ with $(Y^\tau)'\cap Y'=\G_2(4)$. Hence $Z=Y^\tau Y$ with $Y^\tau\cap Y=\GaG_2(4)$.
Moreover, Lemma~\ref{LemSymplectic10} shows that $Y^\tau=X(Y^\tau\cap Y)$ with $X\cap(Y^\tau\cap Y)=5$ or $3$ according to $X=(\SL_2(16)\times5).4$ or $\GaL_2(16)\times3$, respectively. Thus $Z=XY$ with $X\cap Y=5$ or $3$, respectively, in these two cases.
\end{proof}

We are now ready to give the main result of this section.

\begin{lemma}\label{LemOmegaPlus42}
Let $L=\POm_8^+(q)$ with $q\geqslant4$, let $G$ be an almost simple group with socle $L$, let $(A^\tau)^{(\infty)}\leqslant K\leqslant A^\tau$ for some subgroup $A$ of $G$ of the form $A=S\times\calO$, where $S$ is almost simple with socle $\Omega_7(q)$ and $\calO\leqslant(2,q)$, and let $H$ be a subgroup of $A$ with a unique nonsolvable composition factor. Then $G=HK$ if and only if $G/L=(HL/L)(KL/L)$ and one of the following holds:
\begin{enumerate}[{\rm (a)}]
\item $q$ is even, $H$ is a subgroup of $M:=\Pa_1[S]\times\calO$ such that $HK\supseteq\bfO_2(M)$ and $H\bfO_2(M)/\bfO_2(M)$ contains the field-extension subgroup $\Sp_2(q^2)$ of $M^{(\infty)}/\bfO_2(M^{(\infty)})=\Sp_4(q)$;
\item $H^{(\infty)}=\Omega_6^+(q)$, $\Omega_6^-(q)$, $\Omega_5(q)$, $q^5{:}\Omega_5(q)$ ($q$ odd) or $q^4{:}\Omega_4^-(q)$ ($q$ odd);
\item $q=4$, and $(H,K)$ tightly contains some $(X,Y)$ as in Lemma~$\ref{LemOmegaPlus40}$.
\end{enumerate}
\end{lemma}

\begin{proof}
Since $K^{(\infty)}=(A^{(\infty)})^\tau\cong\Omega_7(q)$, Lemma~\ref{LemOmegaPlus26} shows that $L=A^{(\infty)}K^{(\infty)}$ with $A^{(\infty)}\cap K^{(\infty)}=\G_2(q)$. Since $H\leqslant A$, we have $G/L=(HL/L)(KL/L)=(AL/L)(KL/L)$. It follows that $G=AK$ with $(A\cap K)^{(\infty)}=\G_2(q)$.
Now $G=HK$ if and only if $A=H(A\cap K)$.

First assume that $q$ is odd. In this case, $A$ is an almost simple group with socle $\Omega_7(q)$. We conclude from~\cite[Theorem~1.2]{LWX-Omega} that $A=H(A\cap K)$ if and only if $H^{(\infty)}=\Omega_6^+(q)$, $\Omega_6^-(q)$, $\Omega_5(q)$, $q^5{:}\Omega_5(q)$ or $q^4{:}\Omega_4^-(q)$. Hence $G=HK$ if and only if part~(b) of the lemma holds.

Next assume that $q$ is even. For $q=4$, computation in \magma~\cite{BCP1997} shows that $G=HK$ if and only if one of parts~(a)--(c) of the lemma holds.
Thus assume $q\geqslant8$. If part~(a) of the lemma appears, then $HK\supseteq\bfO_2(M)\supseteq\calO$ and Lemma~\ref{LemSymplectic12} implies $A/\calO=(H\calO/\calO)((A\cap K)\calO/\calO)$, which in conjunction with Lemma~\ref{LemXia01} leads to $A=H(A\cap K)$ and hence $G=HK$. If part~(b) of the lemma appears, then we derive from Lemma~\ref{LemSymplectic12} that $H(A\cap K)\supseteq H^{(\infty)}(A\cap K)^{(\infty)}=A^{(\infty)}$ and so $HK\supseteq H(A\cap K)K^{(\infty)}\supseteq A^{(\infty)}K^{(\infty)}=L$, which together with the condition $G/L=(HL/L)(KL/L)$ yields $G=HK$ by Lemma~\ref{LemXia01}.
Conversely, if $G=HK$, then $A=H(A\cap K)$ and so $A/\calO=(H\calO/\calO)((A\cap K)\calO/\calO)$, whence by Lemma~\ref{LemSymplectic12}, either~(a) or~(b) of the lemma appears. This shows that $G=HK$ if and only if~(a) or~(b) of the lemma holds, completing the proof.
\end{proof}

\begin{remark}
From the proof of Lemma~\ref{LemOmegaPlus42} we see that the intersection $H^{(\infty)}\cap K^{(\infty)}$ for the group $H^{(\infty)}$ as described in part~(b) lies in the table:
\[
\begin{array}{|c|ccccc|}
\hline
H^{(\infty)} & \Omega_6^+(q) & \Omega_6^-(q) & \Omega_5(q) & q^5{:}\Omega_5(q) & q^4{:}\Omega_4^-(q) \\
H^{(\infty)}\cap K^{(\infty)} & \SL_3(q) & \SU_3(q) & \SL_2(q) & [q^5]{:}\SL_2(q) & [q^3] \\
\hline
\end{array}
\]
\end{remark}

\section{Infinite families of $(X,Y)$ in Table~\ref{TabOmegaPlus}}

In this section, we construct the infinite families of pairs $(X,Y)$ in Table~\ref{TabOmegaPlus}, namely, the pairs $(X,Y)$ in rows~1--14 of Table~\ref{TabOmegaPlus}. Recall the notation introduced in Section~\ref{SecOmegaPlus01}.

\begin{lemma}\label{LemOmegaPlus01}
Let $G=\Omega(V)=\Omega_{2m}^+(q)$, let $M=R{:}T$, and let $K=G_{e_1+f_1}$. Then the following statements hold:
\begin{enumerate}[{\rm (a)}]
\item the kernel of $M\cap K$ acting on $U$ is $R\cap K=q^{(m-1)(m-2)/2}$;
\item the induced group by the action of $M\cap K$ on $U$ is $\SL(U)_{U_1,e_1+U_1}=q^{m-1}{:}\SL_{m-1}(q)$;
\item $T\cap K=\SL_{m-1}(q)$;
\item if $H=R{:}S$ with $S\leqslant T$, then $H\cap K=(R\cap K).S_{U_1,e_1+U_1}$.
\end{enumerate}
\end{lemma}

\begin{proof}
We first calculate $R\cap K$, the kernel of $M\cap K$ acting on $U$. For each $r\in R\cap K$, since $r$ fixes $e_1$ and $e_1+f_1$, we deduce that $r$ fixes $\langle e_1,e_1+f_1\rangle_{\bbF_q}=\langle e_1,f_1\rangle_{\bbF_q}$ pointwise. Hence $R\cap K$ is isomorphic to the pointwise stabilizer of $U_1$ in $\Omega(\langle e_2,f_2,\dots,e_m,f_m\rangle_{\bbF_q})$, and so $R\cap K=q^{(m-1)(m-2)/2}$.

As $K$ fixes $e_1+f_1$, it stabilizes $(e_1+f_1)^\perp$. Thus $M\cap K$ stabilizes $U\cap(e_1+f_1)^\perp=U_1$.
For arbitrary $h\in M\cap K$, write $e_1^h=\zeta e_1+e$ with $\zeta\in\bbF_q$ and $e\in U_1$. Then
\[
\zeta=\beta(\zeta e_1+e,e_1+f_1)=\beta(e_1^h,(e_1+f_1)^h)=\beta(e_1,e_1+f_1)=1.
\]
This means that $M\cap K$ stabilizes $e_1+U_1$. Hence the induced group of $M\cap K$ on $U$ is contained in $\SL(U)_{U_1,e_1+U_1}$, that is, $(M\cap K)^U\leqslant\SL(U)_{U_1,e_1+U_1}$. Now
\[
(M\cap K)/(R\cap K)\cong(M\cap K)^U\leqslant\SL(U)_{U_1,e_1+U_1}=q^{m-1}{:}\SL_{m-1}(q),
\]
while
\begin{align*}
|M\cap K|\geqslant\frac{|M||K|}{|G|}&=\frac{|q^{m(m-1)/2}{:}\SL_m(q)||\Omega_{2m-1}(q)|}{|\Omega_{2m}^+(q)|}\\
&=q^{(m-1)(m-2)/2}|\SL_{m-1}(q)|=|R\cap K||\SL(U)_{U_1,e_1+U_1}|.
\end{align*}
Thus $(M\cap K)^U=\SL(U)_{U_1,e_1+U_1}=q^{m-1}{:}\SL_{m-1}(q)$.

For each $t\in T\cap K$, we have $e_1^t\in U^t=U$ and $f_1^t\in W^t=W$, and then it follows from
\[
e_1^t+f_1^t=(e_1+f_1)^t=e_1+f_1
\]
that $e_1^t=e_1$ and $f_1^t=f_1$. Hence we conclude that
\[
T\cap K=T_{e_1,f_1}=T_{e_1,U_1}=\SL_{m-1}(q).
\]

Finally, let $H=R{:}S$ with $S\leqslant T$. Since $(M\cap K)^U=\SL(U)_{U_1,e_1+U_1}=(M_{U_1,e_1+U_1})^U$, we obtain $M_{U_1,e_1+U_1}=(M\cap K)R$ as $R$ is the kernel of $M$ acting on $U$. This implies that $H_{U_1,e_1+U_1}=(H\cap K)R$, and so
\[
(H\cap K)/(R\cap K)\cong(H\cap K)R/R=H_{U_1,e_1+U_1}/R=S_{U_1,e_1+U_1}R/R\cong S_{U_1,e_1+U_1}.
\]
Therefore, $H\cap K=(R\cap K).S_{U_1,e_1+U_1}$.
\end{proof}

The next five lemmas give the pairs $(X,Y)$ in row~1 of Table~\ref{TabOmegaPlus}. 

\begin{lemma}\label{LemOmegaPlus02}
Let $G=\Omega(V)=\Omega_{2m}^+(q)$, let $H=R{:}S$ with $S\leqslant\SL_a(q^b)\leqslant T$, where $\SL_a(q^b)$ is a field-extension subgroup of $T$ and $S=\SL_a(q^b)$ ($m=ab$), $\Sp_a(q^b)'$ ($m=ab$), $\G_2(q^b)'$ ($m=6b$, $q$ even) or $\SL_2(13)$ ($m=6$, $q=3$), let $K=G_{e_1+f_1}$, and let $(Z,X,Y)=(\overline{G},\overline{H},\overline{K})$. Then
\[
H\cap K=
\begin{cases}
(q^{(m-1)(m-2)/2}.q^{m-b}){:}\SL_{a-1}(q^b)&\textup{if }S=\SL_a(q^b)\\
2^{3+2}{:}\Sp_2(2)&\textup{if }S=\Sp_a(q^b)'\textup{ with }(a,b,q)=(4,1,2)\\
(q^{(m-1)(m-2)/2}.[q^{m-b}]){:}\Sp_{a-2}(q^b)&\textup{if }S=\Sp_a(q^b)'\textup{ with }(a,b,q)\neq(4,1,2)\\
2^{10+2+2}{:}\SL_2(2)&\textup{if }S=\G_2(q^b)'\textup{ with }(m,q)=(6,2)\\
(q^{(m-1)(m-2)/2}.q^{2b+3b}){:}\SL_2(q^b)&\textup{if }S=\G_2(q^b)'\textup{ with }(m,q)\neq(6,2),\\
3^{10+1}&\textup{if }S=\SL_2(13)\textup{ with }(m,q)=(6,3),
\end{cases}
\]
and $Z=XY$ with $Z=\POm_{2m}^+(q)$, $X=\lefthat(q^{m(m-1)/2}{:}S)$ and $Y\cong K=\Omega_{2m-1}(q)$.
\end{lemma}

\begin{proof}
It is clear that $Z=\POm_{2m}^+(q)$, $X=\lefthat(q^{m(m-1)/2}{:}S)$ and $Y\cong K=\Omega_{2m-1}(q)$.
From Lemma~\ref{LemOmegaPlus01} we obtain $H\cap K=q^{(m-1)(m-2)/2}.S_{U_1,e_1+U_1}$.
By~\cite[Lemmas~4.1,~4.2,~4.4 and~5.5]{LWX-Linear} we have
\[
S_{U_1,e_1+U_1}=
\begin{cases}
q^{m-b}{:}\SL_{a-1}(q^b)&\textup{if }S=\SL_a(q^b)\\
2^2{:}\Sp_2(2)&\textup{if }S=\Sp_a(q^b)'\textup{ with }(a,b,q)=(4,1,2)\\
[q^{m-b}]{:}\Sp_{a-2}(q^b)&\textup{if }S=\Sp_a(q^b)'\textup{ with }(a,b,q)\neq(4,1,2)\\
2^{2+2}{:}\SL_2(2)&\textup{if }S=\G_2(q^b)'\textup{ with }(m,q)=(6,2)\\
q^{2b+3b}{:}\SL_2(q^b)&\textup{if }S=\G_2(q^b)'\textup{ with }(m,q)\neq(6,2)\\
3&\textup{if }S=\SL_2(13)\textup{ with }(m,q)=(6,3).
\end{cases}
\]
This proves the conclusion of this lemma on $H\cap K$ and implies that
\[
\frac{|H|}{|H\cap K|}=\frac{|q^{m(m-1)/2}{:}S|}{|q^{(m-1)(m-2)/2}.S_{U_1,e_1+U_1}|}=q^{m-1}(q^m-1)=\frac{|\Omega_{2m}^+(q)|}{|\Omega_{2m-1}(q)|}=\frac{|G|}{|K|}.
\]
Thus $G=HK$, and so $Z=\overline{G}=\overline{H}\,\overline{K}=XY$.
\end{proof}

\begin{lemma}\label{LemOmegaPlus03}
Let $G=\Omega(V)=\Omega_{2m}^+(q)$, let $H=T$, let $K=G_{e_1+f_1}$, and let $(Z,X,Y)=(\overline{G},\overline{H},\overline{K})$.
Then $H\cap K=\SL_{m-1}(q)$, and $Z=XY$ with $Z=\POm_{2m}^+(q)$, $X=\lefthat\SL_m(q)$ and $Y\cong K=\Omega_{2m-1}(q)$.
\end{lemma}

\begin{proof}
It is clear that $Z=\POm_{2m}^+(q)$, $X=\lefthat\SL_m(q)$ and $Y\cong K=\Omega_{2m-1}(q)$.
The conclusion $H\cap K=\SL_{m-1}(q)$ is the statement of Lemma~\ref{LemOmegaPlus01}(c). This implies that
\[
\frac{|H|}{|H\cap K|}=\frac{|\SL_m(q)|}{|\SL_{m-1}(q)|}=q^{m-1}(q^m-1)=\frac{|\Omega_{2m}^+(q)|}{|\Omega_{2m-1}(q)|}=\frac{|G|}{|K|}.
\]
Consequently, $G=HK$. Thus $Z=\overline{G}=\overline{H}\,\overline{K}=XY$.
\end{proof}



\begin{lemma}\label{LemOmegaPlus07}
Let $G=\Omega(V)=\Omega_{2m}^+(q)$ with even $m$, let $H=\SU(V_\sharp)$, let $K=G_{e_1+f_1}$, and let $(Z,X,Y)=(\overline{G},\overline{H},\overline{K})$.
Then $H\cap K=\SU_{m-1}(q)$, and $Z=XY$ with $Z=\POm_{2m}^+(q)$, $X=\lefthat\SU_m(q)$ and $Y\cong K=\Omega_{2m-1}(q)$.
\end{lemma}

\begin{proof}
It is evident that $Z=\POm_{2m}^+(q)$ and $X=\lefthat\SU_m(q)$. Since $e_1+f_1$ is a nonsingular vector in both $V$ (with respect to $Q$) and $V_\sharp$ (with respect to $\beta_\sharp$), we have $Y\cong K=G_{e_1+f_1}=\Omega_{2m-1}(q)$ and
\[
H\cap K=\SU(V_\sharp)\cap G_{e_1+f_1}=\SU(V_\sharp)_{e_1+f_1}=\SU_{m-1}(q).
\]
It follows that
\[
\frac{|H|}{|H\cap K|}=\frac{|\SU_m(q)|}{|\SU_{m-1}(q)|}=q^{m-1}(q^m-1)=\frac{|\Omega_{2m}^+(q)|}{|\Omega_{2m-1}(q)|}=\frac{|G|}{|K|},
\]
and so $G=HK$. Hence $Z=\overline{G}=\overline{H}\,\overline{K}=XY$.
\end{proof}

Recall the definition of $T$ in Section~\ref{SecOmegaPlus01}.

\begin{lemma}\label{LemOmegaPlus44}
Let $G=\Omega(V)=\Omega_{2m}^+(q)$ with even $m$, let $H=\Sp_m(q)<M<G$ with $M=T$ or $\SU(V_\sharp)$, let $K=G_{e_1+f_1}$, and let $(Z,X,Y)=(\overline{G},\overline{H},\overline{K})$. Then $H\cap K=\Sp_{m-2}(q)$, and $Z=XY$ with $Z=\POm_{2m}^+(q)$, $X=\PSp_m(q)$ and $Y\cong K=\Omega_{2m-1}(q)$.
\end{lemma}

\begin{proof}
It is clear that $Z=\POm_{2m}^+(q)$, $X=\PSp_m(q)$ and $Y\cong K=\Omega_{2m-1}(q)$. By Lemmas~\ref{LemOmegaPlus03} and~\ref{LemOmegaPlus07} we have $G=MK$ with $M\cap K=\SL_{m-1}(q)$ or $\SU_{m-1}(q)$ according to $M=T$ or $\SU(V_\sharp)$, respectively. Moreover, by~\cite[Lemma~4.5]{LWX-Linear} and~\cite[Lemma~4.7]{LWX-Unitary} we have $H\cap(M\cap K)=\Sp_{2m-2}(q)$. This leads to $G=HK$ with $H\cap K=\Sp_{2m-2}(q)$, and so $Z=\overline{G}=\overline{H}\,\overline{K}=XY$.
\end{proof}

If $m$ is even, then $\N_1[\POm_{2m}^+(q)]=\Omega_{2m-1}(q)$ by~\cite[Propositions~4.1.6~and~4.1.7]{KL1990}.

\begin{lemma}\label{LemOmegaPlus08}
Let $Z=\POm_{2m}^+(q)$ with even $m$, let $X=\PSp_m(q)<\PSp_2(q)\otimes\PSp_m(q)<Z$, and let $Y=\N_1[Z]=\Omega_{2m-1}(q)$. Then $Z=XY$ with $Y=\Omega_{2m-1}(q)$ and $X\cap Y=\PSp_{m-2}(q).(2,q-1)<\N_2[X]$.
\end{lemma}

\begin{proof}
By~\cite[Lemma~3.4]{LX2019} we have $X\cap Y=\PSp_{m-2}(q).(2,q-1)<\N_2[X]$. Hence
\[
\frac{|X|}{|X\cap Y|}=\frac{|\PSp_m(q)|}{|\PSp_{m-2}(q).(2,q-1)|}=\frac{q^{m-1}(q^m-1)}{(2,q-1)}=\frac{|\POm_{2m}^+(q)|}{|\Omega_{2m-1}(q)|}=\frac{|Z|}{|Y|},
\]
and so $Z=XY$.
\end{proof}

The next five lemmas are devoted to the construction of $(X,Y)$ in rows~2--3 of Table~\ref{TabOmegaPlus}. 
Recall the definition of $\gamma$ and $\xi$ in Section~\ref{SecOmegaPlus01}.



\begin{lemma}\label{LemOmegaPlus11}
Let $Z=\Omega(V){:}\langle\phi\rangle=\Omega_{2m}^+(q){:}f$ with $m=2a$ and $q\in\{2,4\}$, and let $Y=Z_{e_1+f_1}=\GaSp_{2m-2}(q)$.
\begin{enumerate}[{\rm (a)}]
\item If $M=T{:}\langle\rho\rangle=\SL_m(q){:}\langle\rho\rangle$ with $\rho\in\{\phi,\phi\gamma\}$ and $X=\SL_a(q^2){:}(2f)<M$ is the group $H$ defined in~\cite[Lemmas~4.6~or~4.7]{LWX-Linear}, then $Z=XY$ with $X\cap Y=\SL_{a-1}(q^2)$.
\item If $M=\SU(V_\sharp){:}\langle\xi\rangle=\SU_m(q){:}(2f)$ and $X=\SL_a(q^2){:}(2f)<M$ is the group $H$ defined in~\cite[Lemma~4.4]{LWX-Unitary}, then $Z=XY$ with $X\cap Y=\SL_{a-1}(q^2)$.
\item If $q=2$, $M=\SU(V_\sharp)$ and $X=\SL_a(4){:}2<M$ is the group $H$ defined in~\cite[Lemma~4.3]{LWX-Unitary}, then $Z=XY$ with $X\cap Y=\SL_{a-1}(4)$.
\end{enumerate}
\end{lemma}

\begin{proof}
First assume that $M$ and $X$ are as in part~(a). By Lemma~\ref{LemOmegaPlus03} we have $Z=TZ_{e_1+f_1}$ with $T_{e_1+f_1}=T\cap Z_{e_1+f_1}=\SL_{m-1}(q)$.
Hence $Z=T\langle\rho\rangle Z_{e_1+f_1}=MY$ with $M\cap Y=T_{e_1+f_1}{:}\langle\rho\rangle$. Moreover, we conclude from~\cite[Lemmas~4.6~and~4.7]{LWX-Linear} that $M=X(M\cap Y)$ with $X\cap(M\cap Y)=X\cap(T_{e_1+f_1}\langle\rho\rangle)=\SL_{a-1}(q^2)$. Thus $Z=XY$ with $X\cap Y=X\cap(M\cap Y)=\SL_{a-1}(q^2)$.

Next assume that $M$ and $X$ are given in part~(b). From Lemma~\ref{LemOmegaPlus07} we obtain $Z'=M'Y'$ with $M'\cap Y'=\SU_{m-1}(q)$. Thus $Z=MY$ with $M\cap Y=\SU_{m-1}(q).(2f)$. Now~\cite[Lemma~4.4]{LWX-Unitary} implies that $M=X(M\cap Y)$ with $X\cap(M\cap Y)=\SL_{a-1}(q^2)$. Therefore, $Z=XY$ with $X\cap Y=\SL_{a-1}(q^2)$.

Finally, let $M$ and $X$ be given in part~(c) with $q=2$. By Lemma~\ref{LemOmegaPlus07} we have $Z=MY$ with $M\cap Y=\SU_{m-1}(2)$, while by~\cite[Lemma~4.3]{LWX-Unitary} we have $M=X(M\cap Y)$ with $X\cap Y=X\cap(M\cap Y)=\SL_{a-1}(4)$. Hence $Z=XY$ with $X\cap Y=\SL_{a-1}(4)$.
\end{proof}

The next lemma will also be needed for symplectic groups.

\begin{lemma}\label{LemSymplectic13}
Let $G=\GaSp_{2m}(q)$ with $m$ even and $q\in\{2,4\}$, let $H=\GaSp_m(q^2)$ be a field-extension subgroup of $G$, and let $K=\Sp_{2m-2}(q){:}f<\N_2[G]$ such that $K\cap\Soc(G)=\Sp_{2m-2}(q)$. Then $G=HK$ with $H\cap K=\Sp_{m-2}(q^2)<\N_2[H']$.
\end{lemma}

\begin{proof}
It is proved in~\cite[3.2.1(a)]{LPS1990} that $H\cap\N_2[G]=H'\cap\N_2[G]$. Hence $H\cap K=H'\cap K$. Then since $H'<\Soc(G)$ and $K\cap\Soc(G)=K'=\Sp_{2m-2}(q)$, we derive that
\[
H\cap K=H'\cap(K\cap\Soc(G))=H'\cap K'.
\]
Let $E_1',F_1',\dots,E_{m/2}',F_{m/2}'$ be a standard basis for a symplectic space over $\bbF_{q^2}$ on which $H'=\Sp_m(q^2)$ is defined.
Then $\Soc(G)=\Sp_{2m}(q)$ is defined on the symplectic space $\bbF_q$ with a basis $E_1',\mu E_1',F_1',\mu F_1',\dots,E_{m/2}',\mu E_{m/2}',F_{m/2}',\mu F_{m/2}'$, where $\mu\in\bbF_{q^2}\setminus\bbF_q$, and we may assume without loss of generality that $K'=\Soc(G)_{E_1',F_1'}$. Therefore,
\[
H\cap K=H'\cap K'=(H')_{E_1',F_1'}=\Sp_{m-2}(q^2)<\N_2[H'],
\]
and it follows that
\[
\frac{|H|}{|H\cap K|}=\frac{|\Sp_m(q^2){:}(2f)|}{|\Sp_{m-2}(q^2)|}=2fq^{2m-2}(q^{2m}-1)=q^{2m-1}(q^{2m}-1)=\frac{|\GaSp_{2m}(q)|}{|\Sp_{2m-2}(q){:}f|}=\frac{|G|}{|K|}
\]
as $q\in\{2,4\}$. Thus $G=HK$.
\end{proof}

\begin{remark}
In fact, the conclusion of Lemma~\ref{LemSymplectic13} also holds for $m=2$.
\end{remark}

\begin{lemma}\label{LemOmegaPlus10}
Let $Z=\Omega(V){:}\langle\phi\rangle=\Omega_{2m}^+(q){:}f$ with $m=4a\geqslant8$ and $q\in\{2,4\}$, let $X=\Sp_{2a}(q^2){:}(2f)<\Sp_m(q){:}f<(\Sp_2(q)\otimes\Sp_m(q)){:}f<Z$ such that $X\cap\Omega(V)=\Sp_{2a}(q^2){:}2$, and let $Y=Z_{e_1+f_1}=\N_1[Z]=\GaSp_{2m-2}(q)$. Then $Z=XY$ with $X\cap Y=\Sp_{2a-2}(q^2)<\N_2[X']$.
\end{lemma}

\begin{proof}
Let $M=\Sp_m(q){:}f<(\Sp_2(q)\otimes\Sp_m(q)){:}f<Z$ such that $X<M$. From $X\cap\Omega(V)=\Sp_{2a}(q^2){:}2$ we deduce that $M\Omega(V)/\Omega(V)\geqslant X\Omega(V)/\Omega(V)=f$ and thus $M\Omega(V)=Z$. This implies $M\cap\Omega(V)=\Sp_m(q)$ and $M\cong\GaSp_m(q)$. Then by Lemmas~\ref{LemXia01} and~\ref{LemOmegaPlus08} we have $Z=MY$ with $M\cap Y=\Sp_{m-2}(q).f$. Now, by Lemma~\ref{LemSymplectic13}, $M=X(M\cap Y)$ with $X\cap(M\cap Y)=\Sp_{2a-2}(q^2)<\N_2[X']$. Hence $Z=XY$ with $X\cap Y=X\cap(M\cap Y)<\N_2[X']$.
\end{proof}



Recall the definition of $\psi$ and $r'_w$ in Section~\ref{SecOmegaPlus01}.

\begin{lemma}\label{LemOmegaPlus12}
Let $Z=\Omega(V){:}\langle\phi\rangle=\Omega_{2m}^+(q){:}f$ with $q\in\{2,4\}$ and $m$ even, let $X=\Omega(V_\sharp){:}\langle\rho\rangle$ with $\rho\in\{\psi,\psi r'_{E_1+F_1}\}$, and let $Y=Z_{e_1+f_1}$. Then $Z=XY$ with $X=\Omega_m^+(q^2){:}(2f)$, $Y=\GaSp_{2m-2}(q)$ and $X\cap Y=\Omega_{m-1}(q^2)$.
\end{lemma}

\begin{proof}
It is clear that $X=\Omega_m^+(q^2){:}(2f)$ and $Y=\GaSp_{2m-2}(q)$. Let $S=\Omega(V_\sharp)$.
Since $e_1+f_1$ is nonsingular in both $V$ and $V_\sharp$, we have
\[
S\cap Y=\Omega(V_\sharp)\cap Z_{e_1+f_1}=\Omega(V_\sharp)_{e_1+f_1}=\Omega_{m-1}(q^2).
\]
$S\cap K^{(\infty)}=H^{(\infty)}\cap K^{(\infty)}=\SL_{m-1}(q^2)=S\cap K$. Now as $q\in\{2,4\}$ we have $q=2f$, and so
\[
\frac{|X|}{|S\cap Y|}=\frac{|\Omega_m^+(q^2){:}(2f)|}{|\Sp_{m-2}(q^2)|}=2fq^{m-2}(q^m-1)=q^{m-1}(q^m-1)=\frac{|\Omega_{2m}^+(q){:}f|}{|\GaSp_{2m-2}(q)|}=\frac{|Z|}{|Y|}.
\]
Thus it suffices to prove $X\cap Y=S\cap Y$, or equivalently, $(X\setminus S)\cap Y=\emptyset$.

Suppose for a contradiction that there exists $y\in(X\setminus S)\cap Y$. Since $\rho$ has order $2f$, we have
\[
X=S{:}\langle\rho\rangle=S\cup\rho S\cup\dots\cup\rho^{2f-1}S.
\]
If $y\in\rho^fS$, then $y=\rho^fs$ for some $s\in S$.
If $y\notin\rho^fS$, then $q=4$ and $y^2\in\rho^fS$, which means that $y^2=\rho^fs$ for some $s\in S$.
In either case, there exists $s\in S$ such that $\rho^fs\in Y$. Since $Y=Z_{e_1+f_1}=Z_{\lambda E_1+F_1}$, it follows that
\[
\lambda E_1+F_1=(\lambda E_1+F_1)^{\rho^fs}=(\lambda E_1+F_1)^{\psi^f\psi^f\rho^fs}=(\lambda^qE_1+F_1)^{\psi^f\rho^fs}.
\]
Then as $\psi^f\rho^fs\in\{s,(r'_{E_1+F_1})^fs\}\subset\mathrm{O}(V_\sharp)$, we obtain
\[
\lambda=Q_\sharp(\lambda E_1+F_1)=Q_\sharp(\lambda^qE_1+F_1)=\lambda^q,
\]
contradicting the condition $\lambda^q+\lambda=1$.
\end{proof}

Recall the definition of $\xi$ in Section~\ref{SecOmegaPlus01}.

\begin{lemma}\label{LemOmegaPlus45}
Let $Z=\Omega(V){:}\langle\phi\rangle=\Omega_{2m}^+(q){:}f$ with $m=4a$ and $q\in\{2,4\}$, and let $Y=Z_{e_1+f_1}=\GaSp_{2m-2}(q)$.
\begin{enumerate}[{\rm (a)}]
\item If $M=T{:}\langle\rho\rangle=\SL_m(q){:}\langle\rho\rangle$ with $\rho\in\{\phi,\phi\gamma\}$ and $X=\Sp_{2a}(q^2){:}(2f)<M$ is the group $H$ defined in~\cite[Lemmas~4.8~or~4.9]{LWX-Linear}, then $Z=XY$ with $X\cap Y=\Sp_{2a-2}(q^2)$.
\item If $M=\SU(V_\sharp){:}\langle\xi\rangle=\SU_m(q){:}(2f)$ and $X=\Sp_{2a}(q^2){:}(2f)<M$ is the group $H$ defined in~\cite[Lemma~4.6]{LWX-Unitary}, then $Z=XY$ with $X\cap Y=\Sp_{2a-2}(q^2)$.
\item If $q=2$, $M=\SU(V_\sharp)$ and $X=\Sp_{2a}(4){:}2<M$ is the group $H$ defined in~\cite[Lemma~4.5]{LWX-Unitary}, then $Z=XY$ with $X\cap Y=\Sp_{2a-2}(4)$.
\item If $M=\Omega(V_\sharp){:}\langle\psi\rangle=\Omega_m^+(q^2){:}(2f)$ and $X=\SU_{2a}(q^2){:}(4f)<M$, then $Z=XY$ with $X\cap Y=\SU_{2a-1}(q^2).2$.
\item If $M=\SU_{2a}(q^2){:}(4f)<\Omega(V_\sharp){:}\langle\psi\rangle=\Omega_m^+(q^2){:}(2f)$ and $X=\Sp_{2a}(q^2){:}(4f)<M$, then $Z=XY$ with $X\cap Y=\Sp_{2a-2}(q^2).2$.
\end{enumerate}
\end{lemma}

\begin{proof}
Based on~\cite[Lemmas~4.8~and~4.9]{LWX-Linear} and~\cite[Lemmas~4.5~and~4.6]{LWX-Unitary}, similar arguments as in the proof of Lemma~\ref{LemOmegaPlus11} lead to parts~(a)--(c) of the lemma.

Assume that $M$ and $X$ are as in part~(d). According to Lemma~\ref{LemOmegaPlus12}, $Z=MY$ with $M\cap Y=\Omega_{m-1}(q^2)$. Moreover, Lemma~\ref{LemOmegaPlus07} implies that $M=X(M\cap Y)$ with $X\cap(M\cap Y)=\SU_{2a-1}(q^2).2$. Therefore, $Z=XY$ with $X\cap Y=\SU_{2a-1}(q^2).2$.

Now assume that $M$ and $X$ are as in part~(e). By part~(d) we have $Z=MY$ with $M\cap Y=\SU_{2a-1}(q^2).2$. Moreover,~\cite[Lemma~4.7]{LWX-Unitary} implies that $M=X(M\cap Y)$ with $X\cap(M\cap Y)=\Sp_{2a-2}(q^2).2$. Thus $Z=XY$ with $X\cap Y=\Sp_{2a-2}(q^2).2$.
\end{proof}





Next we construct the pair $(X,Y)$ in row~5 of Table~\ref{TabOmegaPlus}.

\begin{lemma}\label{LemOmegaPlus06}
Let $Z=\Omega(V)=\Omega_{12}^+(q)$ with $m=6$ and $q$ even, let $X=\G_2(q)<\Sp_6(q)<M<Z$ with $M=T$, $\SU(V_\sharp)$ or $\Sp_2(q)\otimes\Sp_6(q)$, and let $Y=Z_{e_1+f_1}$. Then $Z=XY$ with $Y=\Sp_{10}(q)$ and $X\cap Y=\SL_2(q)$.
\end{lemma}

\begin{proof}
Clearly, $Y=\N_1[Z]=\Sp_{10}(q)$. If $M=T$ or $\SU(V_\sharp)$, then Lemmas~\ref{LemOmegaPlus03} and~\ref{LemOmegaPlus07} show that $M\cap Y=\SL_5(q)$ or $\SU_5(q)$, respectively, and so by~\cite[Lemma~4.11]{LWX-Linear} and~\cite[Lemma~4.7]{LWX-Unitary} we have $X\cap(M\cap Y)=\SL_2(q)$.
If $M=\Sp_2(q)\otimes\Sp_6(q)$, then $X<M_1<M$ with $M_1=\Sp_6(q)$, and by Lemma~\ref{LemOmegaPlus08} and~\cite[Lemma~4.10]{LWX-Linear} we have $X\cap(M_1\cap Y)=\SL_2(q)$. In either case, it follows that $X\cap Y=\SL_2(q)$ and so
\[
\frac{|X|}{|X\cap Y|}=\frac{|\G_2(q)|}{|\SL_2(q)|}=q^5(q^6-1)=\frac{|\Omega_{12}^+(q)|}{|\Sp_{10}(q)|}=\frac{|Z|}{|Y|}.
\]
Hence $Z=XY$.
\end{proof}

\begin{remark}
If we let $X=\G_2(q)'$ in Lemma~\ref{LemOmegaPlus06} then the conclusion $Z=XY$ would not hold for $q=2$.
\end{remark}

Recall the notation $u=e_1+e_2+\mu f_2$ introduced in Section~\ref{SecOmegaPlus01}. The following lemma gives the pair $(X,Y)$ in row~7 of Table~\ref{TabOmegaPlus}.

\begin{lemma}\label{LemOmegaPlus13}
Let $G=\Omega(V)=\Omega_{2m}^+(q)$, let $H=R{:}T$, let $K=G_{e_1+f_1,u}$, and let $(Z,X,Y)=(\overline{G},\overline{H},\overline{K})$. Then $T\cap K=\SL_{m-2}(q)$, $H\cap K=(q^{(m-2)(m-3)/2}.q^{2m-4}){:}\SL_{m-2}(q)$, and $Z=XY$ with $Z=\POm_{2m}^+(q)$, $X=\lefthat(q^{m(m-1)/2}{:}\SL_m(q))$ and $Y\cong K=\Omega_{2m-2}^-(q)$.
\end{lemma}

\begin{proof}
It is evident that $Z=\POm_{2m}^+(q)$ and $X=\lefthat(q^{m(m-1)/2}{:}\SL_m(q))$.
Since $\langle e_1+f_1,u\rangle_{\bbF_q}$ is a nondegenerate $2$-subspace of minus type in $V$, we obtain $Y\cong K=\Omega_{2m-2}^-(q)$.
For each $t\in T\cap K$, we have $e_1^t\in U^t=U$, $f_1^t\in W^t=W$, $(e_1+e_2)^t\in U^t=U$ and $(\mu f_2)^t\in W^t=W$. Then it follows from
\[
e_1^t+f_1^t=(e_1+f_1)^t=e_1+f_1\ \text{ and }\ (e_1+e_2)^t+(\mu f_2)^t=u^t=(e_1+e_2)+\mu f_2
\]
that $e_1^t=e_1$, $f_1^t=f_1$, $(e_1+e_2)^t=e_1+e_2$ and $(\mu f_2)^t=\mu f_2$. This implies that $t$ fixes $e_1,e_2,f_1,f_2$.
Hence $T\cap K=T_{e_1,e_2,f_1,f_2}=\SL_{m-2}(q)$.

Next we calculate $R\cap K$, the kernel of $H\cap K$ acting on $U$. For each $r\in R\cap K$, since $r$ fixes $e_1$, $e_2$, $e_1+f_1$ and $u$, we deduce that $r$ fixes $\langle e_1,e_2,e_1+f_1,u\rangle_{\bbF_q}=\langle e_1,e_2,f_1,f_2\rangle_{\bbF_q}$ pointwise. Hence $R\cap K$ is isomorphic to the pointwise stabilizer of $U_2$ in $\Omega(\langle e_3,f_3,\dots,e_m,f_m\rangle_{\bbF_q})$, and so $R\cap K=q^{(m-2)(m-3)/2}$.

Now we consider the action of $H\cap K$ on $U$. Since $K\leqslant G_{e_1+f_1}$, Lemma~\ref{LemOmegaPlus01}(b) shows that $H\cap K$ stabilizes $U_1$ and $e_1+U_1$. Since $K$ fixes $u$, it stabilizes $u^\perp$. Hence $H\cap K$ stabilizes $U_1\cap u^\perp=\langle e_3,\dots,e_m\rangle_{\bbF_q}=U_2$.
For arbitrary $h\in H\cap K$, write $e_1^h=e_1+\zeta e_2+e$ and $e_2^h=\eta e_2+e'$ with $\zeta,\eta\in\bbF_q$ and $e,e'\in U_2$ (note that $h$ stabilizes $e_1+U_1$ and $U_2$). Then
\begin{align*}
&\zeta\mu=\beta(e_1+\zeta e_2+e,e_1+e_2+\mu f_2)=\beta(e_1^h,(e_1+e_2+\mu f_2)^h)=\beta(e_1,e_1+e_2+\mu f_2)=0,\\
&\eta\mu=\beta(\eta e_2+e',e_1+e_2+\mu f_2)=\beta(e_2^h,(e_1+e_2+\mu f_2)^h)=\beta(e_2,e_1+e_2+\mu f_2)=\mu,
\end{align*}
and so $\zeta=0$ and $\eta=1$. Therefore,
\[
(H\cap K)^U\leqslant\left\{
\begin{pmatrix}
1&&A_{13}\\
&1&A_{23}\\
&&A_{33}
\end{pmatrix}
\,\middle|\,A_{13},A_{23}\in\bbF_q^{m-2},\A_{33}\in\SL_{m-2}(q)\right\}.
\]
It follows that
\[
(H\cap K)/(R\cap K)\cong(H\cap K)^U\leqslant q^{2m-4}{:}\SL_{m-2}(q),
\]
which combined with the fact
\begin{align*}
|H\cap K|\geqslant\frac{|H||K|}{|G|}&=\frac{|q^{m(m-1)/2}{:}\SL_m(q)||\Omega_{2m-2}^-(q)|}{|\Omega_{2m}^+(q)|}\\
&=q^{(m-2)(m+1)/2}|\SL_{m-2}(q)|=|R\cap K||q^{2m-4}{:}\SL_{m-2}(q)|
\end{align*}
yields $H\cap K=(R\cap K).(q^{2m-4}{:}\SL_{m-2}(q))=(q^{(m-2)(m-3)/2}.q^{2m-4}){:}\SL_{m-2}(q)$ and $|H\cap K|=|H||K|/|G|$. As a consequence, $G=HK$, and so $Z=\overline{G}=\overline{H}\,\overline{K}=XY$.
\end{proof}

In the next lemma we construct the pair $(X,Y)$ in row~8 of Table~\ref{TabOmegaPlus}. Recall the definition of $\gamma$ in Section~\ref{SecOmegaPlus01}.

\begin{lemma}\label{LemOmegaPlus14}
Let $Z=\langle\Omega(V),\gamma\rangle$ with $q=2$, let $X=T{:}\langle\gamma\rangle$, and let $Y=Z_{e_1+f_1,u}$. Then $Z=XY$ with $Z=\Omega_{2m}^+(2).(2,m-1)$,
$X=\SL_m(2).2$, $Y=\Omega_{2m-2}^-(2).(2,m-1)$ and $X\cap Y=\SL_{m-2}(2)$.
\end{lemma}

\begin{proof}
Clearly, $X=\SL_m(2).2$. As $\gamma=r_{e_1+f_1}\cdots r_{e_m+f_m}$ is a product of $m$ reflections, we have $\gamma\in\Omega(V)$ if and only if $m$ is even.
Hence $Z=\Omega_{2m}^+(2).(2,m-1)$. Then since $\langle e_1+f_1,u\rangle_{\bbF_q}$ is a nondegenerate $2$-subspace of minus type in $V$, we obtain $Y=\Omega_{2m-2}^-(2).(2,m-1)$. Moreover, $T\cap Y=T\cap(\Omega(V)\cap Y)=T\cap\Omega(V)_{e_1+f_1,u}=\SL_{m-2}(2)$ by Lemma~\ref{LemOmegaPlus13}.

Suppose that $X\cap Y\neq T\cap Y$, that is, there exists $t\in T$ with $\gamma t\in Y=Z_{e_1+f_1,u}$. Then
\begin{align*}
&e_1+f_1=(e_1+f_1)^{\gamma t}=(f_1+e_1)^t=e_1^t+f_1^t,\\
&(e_1+e_2)+\mu f_2=u=u^{\gamma t}=(e_1+e_2+\mu f_2)^{\gamma t}=(f_1+f_2+\mu e_2)^t=\mu e_2^t+(f_1+f_2)^t,
\end{align*}
and thereby we deduce $e_1=e_1^t$ and $\mu f_2=(f_1+f_2)^t$ as $t\in T$ stabilizes $U$ and $W$. This yields
\[
0=\beta(e_1,\mu f_2)=\beta(e_1^t,(f_1+f_2)^t)=\beta(e_1,f_1+f_2)=1,
\]
a contradiction.

Thus we conclude that $X\cap Y=T\cap Y=\SL_{m-2}(2)$. Consequently,
\[
\frac{|X|}{|X\cap Y|}=\frac{|\SL_m(2).2|}{|\SL_{m-2}(2)|}=2^{2m-2}(2^m-1)(2^{m-1}-1)=\frac{|\Omega_{2m}^+(2).(2,m-1)|}{|\Omega^-_{2m-2}(2).(2,m-1)|}=\frac{|Z|}{|Y|},
\]
which implies $Z=XY$.
\end{proof}

The pairs $(X,Y)$ in rows~9 and~10 of Table~\ref{TabOmegaPlus} will be constructed in the following two lemmas respectively.

\begin{lemma}\label{LemOmegaPlus15}
Let $Z=\Omega(V)=\Omega_{2m}^+(2)$ with $q=2$, let $X=T$, and let $Y=Z_{\{e_1+f_1+u,u\}}$. Then $Z=XY$ with $X=\SL_m(2)$, $Y=\Omega_{2m-2}^-(2).2$ and $X\cap Y=\SL_{m-2}(2)$.
\end{lemma}

\begin{proof}
It is clear that $X=\SL_m(2)$. Since $\langle e_1+f_1+u,u\rangle_{\bbF_q}=\langle e_1+f_1,u\rangle_{\bbF_q}$ is a nondegenerate $2$-subspace of minus type in $V$ and $Q(e_1+f_1+u)=\mu=Q(u)$, we have $Z_{e_1+f_1+u,u}=\Omega_{2m-2}^-(2)$ and $Z_{\{e_1+f_1+u,u\}}=Z_{e_1+f_1+u,u}{:}\langle s\rangle$ for some $s$ swapping $e_1+f_1+u$ and $u$. In particular, $X=\Omega_{2m-2}^-(2).2$. As Lemma~\ref{LemOmegaPlus13} asserts, $T\cap Z_{e_1+f_1,u}=\SL_{m-2}(2)$.

Suppose that $T_{\{e_1+f_1+u,u\}}\neq T_{e_1+f_1+u,u}$, that is, there exists $t\in T$ swapping $e_1+f_1+u$ and $u$. Then
\begin{align*}
&e_2^t+(f_1+\mu f_2)^t=(e_2+f_1+\mu f_2)^t=(e_1+e_2+u)^t=u=(e_1+e_2)+\mu f_2,\\
&(e_1+e_2)^t+(\mu f_2)^t=(e_1+e_2+\mu f_2)^t=u^t=e_1+f_1+u=e_2+(f_1+\mu f_2),
\end{align*}
from which we obtain $e_2^t=e_1+e_2$ and $(\mu f_2)^t=f_1+\mu f_2$ as $t\in T$ stabilizes $U$ and $W$. However, this leads to
\[
\mu=\beta(e_2,\mu f_2)=\beta(e_2^t,(\mu f_2)^t)=\beta(e_1+e_2,f_1+\mu f_2)=1+\mu,
\]
a contradiction.

Thus we conclude that $T_{\{e_1+f_1+u,u\}}=T_{e_1+f_1+u,u}$. Consequently,
\[
X\cap Y=T\cap Z_{\{e_1+f_1+u,u\}}=T\cap Z_{e_1+f_1,u}=\SL_{m-2}(2).
\]
It follows that
\[
\frac{|X|}{|X\cap Y|}=\frac{|\SL_m(2)|}{|\SL_{m-2}(2)|}=2^{2m-3}(2^m-1)(2^{m-1}-1)=\frac{|\Omega_{2m}^+(2)|}{|\Omega^-_{2m-2}(2).2|}=\frac{|Z|}{|Y|},
\]
and so $Z=XY$.
\end{proof}

Recall the notation $u'=(1-\mu^2)e_1+(\mu^2-1+\mu^{-1})e_2+\mu^2f_1+\mu^2f_2$ introduced in Section~\ref{SecOmegaPlus01}. Also recall that we have some $y\in\mathrm{O}(V)$ with $(e_1+f_1)^y=e_1+f_1$ and $u^y=u'$, and such $y$ exists both in $\Omega(V)$ and $\mathrm{O}(V)\setminus\Omega(V)$.

\begin{lemma}\label{LemOmegaPlus16}
Let $Z=\Omega(V){:}\langle\rho\rangle$ with $q=4$ and $\rho\in\{\phi,\gamma\phi\}$, let $X=T{:}\langle\rho\rangle$, and let $Y=\langle\Omega(V)_{e_1+f_1,u},y\phi\rangle$ such that $y\phi\in Z$. Then $Z=XY$ with $Z=\Omega_{2m}^+(4).2$, $X=\SL_m(4).2$, $Y=\Omega_{2m-2}^-(4).4$ and $X\cap Y=\SL_{m-2}(4)$.
\end{lemma}

\begin{proof}
It is clear that $Z=\Omega_{2m}^+(4).2$ and $X=\SL_m(4).2$. Let $M=Z_{\langle e_1+f_1,u\rangle_{\bbF_4}}$ and $S=\Omega(V)_{e_1+f_1,u}$. Then $S=\Omega_{2m-2}^-(4)$ as $\langle e_1+f_1,u\rangle_{\bbF_4}$ is a nondegenerate $2$-subspace of minus type in $V$. Note that $\mu\notin\{0,1\}$ as $x^2+x+\mu$ is irreducible over $\bbF_4$.
Thus $\mu^3=1$ and $\mu^2+\mu+1=(\mu^3-1)/(\mu-1)=0$. Accordingly,
\[
u'=\mu e_1+e_2+\mu^2f_1+\mu^2f_2.
\]
For $a,b\in\bbF_4$ we have
\begin{align}\label{EqnOmegaPlus07}
(a(e_1+f_1)+bu)^{y\phi}&=(a(e_1+f_1)+bu')^\phi\nonumber\\
&=((a+b+b\mu^2)e_1+be_2+(a+b\mu^2)f_1+b\mu^2f_2)^\phi\nonumber\\
&=(a^2+b^2+b^2\mu)e_1+b^2e_2+(a^2+b^2\mu)f_1+b^2\mu f_2\nonumber\\
&=(a^2+b^2\mu)(e_1+f_1)+b^2u\in\langle e_1+f_1,u\rangle_{\bbF_4}.
\end{align}
This implies that $y\phi$ stablizes $\langle e_1+f_1,u\rangle_{\bbF_4}$ and induces a permutation of order $4$ on it. Consequently, $y\phi$ lies in $M$ and hence normalizes $M^{(\infty)}=S$. Since $S$ is the kernel of $M$ acting on $\langle e_1+f_1,u\rangle_{\bbF_4}$, it follows that
\[
Y=\langle S,y\phi\rangle=S\langle y\phi\rangle=S.4=\Omega_{2m-2}^-(4).4.
\]
As Lemma~\ref{LemOmegaPlus13} asserts that $T\cap S=\SL_{m-2}(4)$, we have
\[
\frac{|X|}{|T\cap S|}=\frac{|\SL_m(4).2|}{|\SL_{m-2}(4)|}=2^{4m-5}(4^m-1)(4^{m-1}-1)=\frac{|\Omega_{2m}^+(4).2|}{|\Omega^-_{2m-2}(4).4|}=\frac{|Z|}{|Y|}.
\]
Thus it suffices to prove $X\cap Y=T\cap S$.

Suppose for a contradiction that $X\cap Y\neq T\cap S$, that is, there exist $i\in\{0,1\}$, $j\in\{0,1,2,3\}$ and $z\in T\rho^i\cap S(y\phi)^j$ with $(i,j)\neq(0,0)$. It follows that
\[
z\in\mathrm{O}(V)t\rho^i\cap\mathrm{O}(V)s(y\phi)^j=\mathrm{O}(V)\phi^i\cap\mathrm{O}(V)\phi^j,
\]
and so $(i,j)\in\{(0,2),(1,1),(1,3)\}$. If $(i,j)=(1,1)$ or $(1,3)$, then
\[
z^2\in T\rho^{2i}\cap S(y\phi)^{2j}=T\cap S(y\phi)^2.
\]
Thus for all $(i,j)\in\{(0,2),(1,1),(1,3)\}$, there exist $t\in T$ and $s\in S$ such that $t=s(y\phi)^2$.
In view of~\eqref{EqnOmegaPlus07} and $\mu^2+\mu+1=0$, one obtains
\begin{align*}
&e_1^t+f_1^t=(e_1+f_1)^t=(e_1+f_1)^{s(y\phi)^2}=(e_1+f_1)^{(y\phi)^2}=(e_1+f_1)^{y\phi}=e_1+f_1,\\
&(e_1+e_2)^t+(\mu f_2)^t=u^t=u^{s(y\phi)^2}=u^{(y\phi)^2}=(\mu(e_1+f_1)+u)^{y\phi}=e_2+(f_1+\mu f_2),
\end{align*}
from which we deduce $e_1^t=e_1$ and $(\mu f_2)^t=f_1+\mu f_2$ as $t\in T$ stabilizes $U$ and $W$. This yields
\[
0=\beta(e_1,\mu f_2)=\beta(e_1^t,(\mu f_2)^t)=\beta(e_1,f_1+\mu f_2)=1,
\]
a contradiction, which completes the proof.
\end{proof}

Now we construct the pair $(X,Y)$ in row~11 of Table~\ref{TabOmegaPlus}.

\begin{lemma}\label{LemOmegaPlus17}
Let $G=\Omega(V)=\Omega_{2m}^+(q)$ with even $m$, let $H=G_{e_1}$, let $K=\SU(V_\sharp)$, and let $(Z,X,Y)=(\overline{G},\overline{H},\overline{K})$.
Then $H\cap K=(q.q^{2m-4}){:}\SU_{m-2}(q)$, and $Z=XY$ with $Z=\POm_{2m}^+(q)$, $X\cong H=q^{2m-2}{:}\Omega_{2m-2}^+(q)$ and $Y=\lefthat\SU_m(q)$.
\end{lemma}

\begin{proof}
It is evident that $Z=\POm_{2m}^+(q)$ and $Y=\lefthat\SU_m(q)$. Since $e_1$ is a singular vector in both $V$ (with respect to $Q$) and $V_\sharp$ (with respect to $\beta_\sharp$), we have $X\cong H=G_{e_1}=q^{2m-2}{:}\Omega_{2m-2}^+(q)$ and
\[
H\cap K=G_{e_1}\cap\SU(V_\sharp)=\SU(V_\sharp)_{e_1}=(q.q^{2m-4}){:}\SU_{m-2}(q).
\]
It follows that
\[
\frac{|K|}{|H\cap K|}=\frac{|\SU_m(q)|}{|(q.q^{2m-4}){:}\SU_{m-2}(q)|}=(q^m-1)(q^{m-1}+1)=\frac{|\Omega_{2m}^+(q)|}{|q^{2m-2}{:}\Omega_{2m-2}^+(q)|}=\frac{|G|}{|H|},
\]
and so $G=HK$. Hence $Z=\overline{G}=\overline{H}\,\overline{K}=XY$.
\end{proof}

The pairs $(X,Y)$ in rows~12--14 of Table~\ref{TabOmegaPlus} will be constructed in the following three lemmas. 
Recall the definition of $\xi$ in Section~\ref{SecOmegaPlus01}.

\begin{lemma}\label{LemOmegaPlus19}
Let $Z=\Omega(V){:}\langle\phi\rangle$ with $m$ even and $q\in\{2,4\}$, let $X=Z_{e_1,f_1}$, and let $Y=\SU(V_\sharp){:}\langle\xi\rangle$. Then $Z=XY$ with $Z=\Omega_{2m}^+(q).f$, $X=\Omega_{2m-2}^+(q).f$, $Y=\SU_m(q).(2f)$ and $X\cap Y=\SU_{m-2}(q)$.
\end{lemma}

\begin{proof}
It is evident that $Z=\Omega_{2m}^+(q).f$, $X=\Omega_{2m-2}^+(q).f$, and $Y=\SU_m(q).(2f)$.

Suppose that $X\cap Y\neq X\cap\SU(V_\sharp)$. Then there exists $t\in\SU(V_\sharp)$ such that $\xi^ft\in X=Z_{e_1,f_1}$. This means that $\xi^ft$ fixes $e_1=\lambda E_1$ and $f_1=F_1$, which together with $t\in\SU(V_\sharp)$ yields
\[
\lambda=\beta_\sharp(\lambda E_1,F_1)=\beta_\sharp((\lambda E_1)^{\xi^ft},F_1^{\xi^ft})
=\beta_\sharp((\lambda^qE_1)^t,F_1^t)=\beta_\sharp(\lambda^qE_1,F_1)=\lambda^q,
\]
contradicting the condition $\lambda+\lambda^q=1$.

Thus we conclude that $X\cap Y=X\cap\SU(V_\sharp)$. Accordingly,
\[
X\cap Y=\Omega(V)_{e_1,f_1}\cap\SU(V_\sharp)=\SU(V_\sharp)_{e_1,f_1}=\SU(V_\sharp)_{\lambda E_1,F_1}=\SU_{m-2}(q).
\]
Observe $q=2f$ as $q\in\{2,4\}$. We then derive that
\[
\frac{|Y|}{|X\cap Y|}=\frac{|\SU_m(q).(2f)|}{|\SU_{m-2}(q)|}=2fq^{2m-3}(q^m-1)(q^{m-1}+1)=\frac{|\Omega^+_{2m}(q).f|}{|\Omega_{2m-2}^+(q).f|}=\frac{|Z|}{|X|},
\]
and so $Z=XY$.
\end{proof}

\begin{lemma}\label{LemOmegaPlus18}
Let $Z=\Omega(V)=\Omega_{2m}^+(2)$ with $m$ even and $q=2$, let $X=Z_{\{e_1,f_1\}}$, and let $Y=\SU(V_\sharp)$. Then $Z=XY$ with $X=\Omega_{2m-2}^+(2).2$, $Y=\SU_m(2)$ and $X\cap Y=\SU_{m-2}(2)$.
\end{lemma}

\begin{proof}
Clearly, $Y=\SU_m(2)$. Since $(e_1,f_1)$ is a hyperbolic pair with respect to $Q$, we have $Z_{e_1,f_1}=\Omega_{2m-2}^+(q)$ and $Z_{\{e_1,f_1\}}=Z_{e_1,f_1}{:}\langle s\rangle$ for some $s$ swapping $e_1$ and $f_1$. In particular, $X=\Omega_{2m-2}^+(2).2$.

Suppose that $Y_{\{e_1,f_1\}}\neq Y_{e_1,f_1}$, that is, there exists $t\in Y$ swapping $e_1=\lambda E_1$ and $f_1=F_1$.
Then as $t\in Y=\SU(V_\sharp)$, we obtain
\[
\lambda=\beta_\sharp(\lambda E_1,F_1)=\beta_\sharp((\lambda E_1)^t,F_1^t)=\beta_\sharp(F_1,\lambda E_1)=\lambda^q,
\]
contradicting the condition $\lambda+\lambda^q=1$.

Thus we conclude that $Y_{\{e_1,f_1\}}=Y_{e_1,f_1}$. Consequently,
\[
X\cap Y=Z_{\{e_1,f_1\}}\cap Y=Y_{\{e_1,f_1\}}=Y_{e_1,f_1}=\SU(V_\sharp)_{e_1,f_1}=\SU(V_\sharp)_{\lambda E_1,F_1}=\SU_{m-2}(2).
\]
It follows that
\[
\frac{|Y|}{|X\cap Y|}=\frac{|\SU_m(2)|}{|\SU_{m-2}(2)|}=2^{2m-3}(2^m-1)(2^{m-1}+1)=\frac{|\Omega_{2m}^+(2)|}{|\Omega^+_{2m-2}(2).2|}=\frac{|Z|}{|X|},
\]
which yields $Z=XY$.
\end{proof}

\begin{lemma}\label{LemOmegaPlus22}
Let $Z=\Omega(V){:}\langle\phi\rangle=\Omega_{2m}^+(4){:}2$ with $m$ even and $q=4$, let $X=\Omega(V)_{e_1,f_1}{:}\langle r_{e_1+f_1}r_{e_2+f_2}\phi\rangle$, and let $Y=\SU(V_\sharp){:}\langle\xi\rangle$. Then $Z=XY$ with $X=\Omega_{2m-2}^+(4).2$, $Y=\SU_m(4).4$ and $X\cap Y=\SU_{m-2}(4)$.
\end{lemma}

\begin{proof}
It clear that $Y=\SU_m(4).4$. Since $(e_1,f_1)$ is a hyperbolic pair with respect to $Q$, we have $\Omega(V)_{e_1,f_1}=\Omega_{2m-2}^+(4)$ and so $X=\Omega_{2m-2}^+(4).2$. By Lemma~\ref{LemOmegaPlus19} we obtain
\begin{equation}\label{EqnOmegaPlus08}
\Omega(V)_{e_1,f_1}\cap Y=Z_{e_1,f_1}\cap Y=Z_{e_1,f_1}\cap\SU(V_\sharp)=\SU_{m-2}(4).
\end{equation}

Suppose that $X\cap Y\nleqslant\Omega(V)_{e_1,f_1}$. Since $X=\Omega(V)_{e_1,f_1}{:}\langle r_{e_1+f_1}r_{e_2+f_2}\phi\rangle$, this means that there exists $s\in\Omega(V)_{e_1,f_1}$ with $r_{e_1+f_1}r_{e_2+f_2}\phi s\in Y$. Since $r_{e_1+f_1}r_{e_2+f_2}\phi s\notin\mathrm{O}(V)$ and $\SU(V_\sharp){:}\langle\xi^2\rangle<\mathrm{O}(V)$, it follows that $r_{e_1+f_1}r_{e_2+f_2}\phi s\in\SU(V_\sharp)\{\xi,\xi^3\}$, and so $(r_{e_1+f_1}r_{e_2+f_2}\phi s)^2\in\SU(V_\sharp)\xi^2$. In particular, $(r_{e_1+f_1}r_{e_2+f_2}\phi s)^2\notin\SU(V_\sharp)$. However, as $X=\Omega(V)_{e_1,f_1}.2$, we have $(r_{e_1+f_1}r_{e_2+f_2}\phi s)^2\in\Omega(V)_{e_1,f_1}$. This shows that the element $(r_{e_1+f_1}r_{e_2+f_2}\phi s)^2$ of $Y$ lies in $\Omega(V)_{e_1,f_1}\setminus\SU(V_\sharp)$, contradicting~\eqref{EqnOmegaPlus08}.

Thus $X\cap Y\leqslant\Omega(V)_{e_1,f_1}$, which in conjunction with~\eqref{EqnOmegaPlus08} gives $X\cap Y=\SU_{m-2}(4)$. Hence
\[
\frac{|Y|}{|X\cap Y|}=\frac{|\SU_m(4).4|}{|\SU_{m-2}(4)|}=4^{2m-2}(4^m-1)(4^{m-1}+1)=\frac{|\Omega_{2m}^+(4).2|}{|\Omega^+_{2m-2}(4).2|}=\frac{|Z|}{|X|},
\]
and so $Z=XY$.
\end{proof}

The next result follows from Lemma~B and its proof in~\cite[Page~105]{LPS1990}.

\begin{lemma}\label{LemOmegaPlus28}
Let $Z=\POm_8^+(q)$ with $m=4$ and $q$ square, let $X=\Omega_8^-(q^{1/2})$ be a $\calC_9$-subgroup of $Z$, and let $Y=\N_1[Z]=\Omega_7(q)$. Then $Z=XY$ with $X\cap Y=\G_2(q^{1/2})$.
\end{lemma}

The following result, which is from~\cite[Appendix~3]{LPS1990}, gives the pair $(X,Y)$ in row~6 of Table~\ref{TabOmegaPlus}.

\begin{lemma}\label{LemOmegaPlus29}
Let $Z=\POm_{16}^+(q)$ with $m=8$, let $X=\Omega_9(q)$ be a $\calC_9$-subgroup of $Z$, and let $Y=\N_1[Z]=\Omega_{15}(q)$. Then $Z=XY$ with $X\cap Y=\Omega_7(q).(2,q-1)<\N_2[X]$.
\end{lemma}

\section{Sporadic cases of $(X,Y)$ in Table~\ref{TabOmegaPlus}}\label{SecOmegaPlus03}

Computation in \magma~\cite{BCP1997} gives the following four lemmas.

\begin{lemma}\label{LemOmegaPlus30}
Let $Z=\Omega_8^+(2)$, and let $Y$ be a subgroup of $Z$ isomorphic to $\A_9$, $\Sp_6(2)$ or $\SU_4(2)$. Then for each isomorphism type in the corresponding $X$ rows of the following tables, there are precisely two conjugacy classes of subgroups $X$ of $Z$ such that $Z=XY$ and that the intersection $X\cap Y$ is described in the $X\cap Y$ rows of the tables.

\[
\begin{array}{|c|ccccccc|}
\hline
X & \Sy_5 & \A_5{:}4 & 2^4{:}\A_5 & \A_6 & 2^5{:}\A_6 & \A_7 & 2^6{:}\A_7 \\
Y & \Sp_6(2) & \Sp_6(2) & \Sp_6(2) & \Sp_6(2) & \Sp_6(2) & \Sp_6(2) & \Sp_6(2) \\
X\cap Y & 1 & 2 & \Q_8 & 3 & 4^2{:}3{:}2 & 7{:}3 & 2^3.\SL_3(2) \\
\hline
\end{array}
\]
\[
\begin{array}{|c|cccccc|}
\hline
X & \A_8 & \A_8 & \A_9 & 2^6{:}\A_7  & \A_8 & \A_8 \\
Y & \Sp_6(2) & \Sp_6(2) & \Sp_6(2) & \SU_4(2) & \SU_4(2) & \SU_4(2).2 \\
X\cap Y & \SL_3(2) & \AGaL_1(8) & \PGaL_2(8) & \SL_2(3) & 3 & \Sy_3 \\
\hline
\end{array}
\]
\[
\begin{array}{|c|ccccccc|}
\hline
X & \Sy_8 & \A_9 & 2^4{:}\A_5 & 2^5{:}\A_6 & 2^6{:}\A_7 & \A_8 & 2^6{:}\A_8 \\
Y & \SU_4(2) & \SU_4(2) & \A_9 & \A_9 & \A_9 & \A_9 & \A_9 \\
X\cap Y & \Sy_3 & 9{:}3 & 1 & \A_4 & \SL_3(2) & 7{:}3 & \AGL_3(2) \\
\hline
\end{array}
\]
\end{lemma}

\begin{lemma}\label{LemOmegaPlus31}
Let $Z=\POm_8^+(3)$, and let $Y$ be a subgroup of $Z$ isomorphic to $\Omega_7(3)$, $3^6{:}\PSL_4(3)$ or $\Omega_8^+(2)$. Then for each isomorphism type in the corresponding $X$ rows of the following tables, there are precisely $k$ conjugacy classes of subgroups $X$ of $Z$ such that $Z=XY$, where $k$ is given in the $k$ rows of the tables. For each such pair $(X,Y)$, the intersection $X\cap Y$ is described in the $X\cap Y$ rows of the tables.
\[
\begin{array}{|c|ccccccc|}
\hline
X & 3^4{:}\Sy_5 & 3^4{:}(4\times\A_5) & (3^5{:}2^4){:}\A_5 & \A_9 & \SU_4(2) & \Sp_6(2) & \Omega_8^+(2) \\
k & 4 & 4 & 4 & 4 & 8 & 4 & 2 \\
Y & \Omega_7(3) & \Omega_7(3) & \Omega_7(3) & \Omega_7(3) & \Omega_7(3) & \Omega_7(3) & \Omega_7(3) \\
X\cap Y & 3^2 & 3\times\Sy_3 & \AGL_3(2) & \SL_3(2) & \SL_2(3) & 2^3.\SL_3(2) & 2^6{:}\A_7 \\
\hline
\end{array}
\]
\[
\begin{array}{|c|ccccc|}
\hline
X & 2^6{:}\A_7 & \A_8 & 2^6{:}\A_8 & \A_9 & 2.\PSL_3(4) \\
k & 4 & 8 & 4 & 8 & 2 \\
Y & 3^6{:}\PSL_4(3) & 3^6{:}\PSL_4(3) & 3^6{:}\PSL_4(3) & 3^6{:}\PSL_4(3) & 3^6{:}\PSL_4(3) \\
X\cap Y & (3\times\SL_2(3)){:}2 & 3\times\Sy_3 & (2^3{:}\Sy_4){:}\Sy_3 & 3^3{:}\Sy_3 & 3^2{:}4 \\
\hline
\end{array}
\]
\[
\begin{array}{|c|ccccc|}
\hline
X & \Sp_6(2) & \Omega_6^-(3) & [3^6]{:}\SL_3(3) & 3^{6+3}{:}\SL_3(3) & 3^6{:}\PSL_4(3) \\
k & 8 & 2 & 6 & 3 & 3 \\
Y & 3^6{:}\PSL_4(3) & 3^6{:}\PSL_4(3) & \Omega_8^+(2) & \Omega_8^+(2) & \Omega_8^+(2) \\
X\cap Y & 3.\AGL_2(3) & 3^{3+2}{:}\SL_2(3) & (3\times\SL_2(3)){:}2 & 3^2.\AGL_2(3) & (3\times\PSp_4(3)){:}2 \\
\hline
\end{array}
\]
\end{lemma}

\begin{remark}
There are two isomorphism types of the form $[3^6]{:}\SL_3(3)$ in the above table; one is $3^6{:}\SL_3(3)$ (three conjugacy classes) and the other is $3^{3+3}{:}\SL_3(3)$ (three conjugacy classes).
\end{remark}

\begin{lemma}\label{LemOmegaPlus25}
Let $Z=\Omega(V){:}\langle\phi\rangle=\Omega_8^+(4){:}2$ with $m=q=4$, let $Y=\N_1[Z]=\GaSp_6(4)<Z$, and let $M=\Omega_4^+(16).2^2.2$ be a $\calC_3$-subgroup of $Z$. There are a unique conjugacy class of subgroups $X$ of $M$ isomorphic to $\GaL_2(16)$ and a unique conjugacy class of subgroups $X$ of $M$ isomorphic to $\SL_2(16).8$ such that $Z=XY$. For each such pair $(X,Y)$ we have
\[
X\cap Y=
\begin{cases}
1&\textup{if }X=\GaL_2(16)\\
2&\textup{if }X=\SL_2(16).8.
\end{cases}
\]
\end{lemma}

\begin{lemma}\label{LemOmegaPlus32}
Let $G=\GaO_4^+(q)$ with $q\in\{4,16\}$, let $H=\Nor_G(\SL_2(q^{1/2})\times\SL_2(q^{1/2}))$, and let $K=\GaU_2(q)<G$. Then $H\cap K<\Soc(G)$.
\end{lemma}

\begin{remark}
The conclusion of Lemma~\ref{LemOmegaPlus32} holds for all subgroups $H$ and $K$ of the given form, not only for one of these subgroups in a conjugacy class.
\end{remark}

Lemma~\ref{LemOmegaPlus32} is needed in the proof of the following lemma.

\begin{lemma}\label{LemOmegaPlus33}
Let $Z=\Omega(V){:}\langle\phi\rangle=\Omega_8^+(q){:}f$ with $m=4$ and $q\in\{4,16\}$, let $X=\Omega_8^-(q^{1/2}).f$ be a $\mathcal{C}_5$-subgroup of $Z$, and let $Y=\GU_4(q).(2f)$ be a $\mathcal{C}_3$-subgroup of $Z$. Then $Z=XY$ with $X\cap Y=X\cap Y\cap\Soc(Z)=\SL_2(q^{1/2})\times\D_{2(q+1)}$.
\end{lemma}

\begin{proof}
Let $L=\Soc(Z)$, and let $\sigma$ be the involutory graph automorphism of $L$. Then the conjugacy class of $\mathcal{C}_5$-subgroups $\Omega_8^-(q^{1/2}).f$ of $Z$ is stabilized by $\sigma$, while the two conjugacy classes of $\mathcal{C}_3$-subgroups $\GU_4(q).(2f)$ of $Z$ are swapped by $\sigma$.
Since $\sigma$ commutes with the field automorphism of $L$, we have $Z^\sigma=Z$. Thus to prove the lemma we may take $Y$ to be a subgroup in any of the two conjugacy classes of $\mathcal{C}_3$-subgroups $\GU_4(q).(2f)$ of $Z$.
Hence we may assume that $Y=\GaU(V_\sharp)$ for some $4$-dimensional vector space $V_\sharp$ over $\bbF_{q^2}$ equipped with a nondegenerate Hermitian form $\beta_\sharp$ such that $Q(w)=\beta_\sharp(w,w)$ for all $w\in V$.

Let $J=X\cap Y$ and $I=X\cap Y\cap L$. According to computation in \magma~\cite{BCP1997}, replacing $X$ with some $L$-conjugate of $X$, we have $I=C\times D<\GaU(V_1)\times\GaU(V_2)$ for some nondegenerate $2$-subspaces $V_1$ and $V_2$ of $V_\sharp$, perpendicular to each other, such that
\[
C=I^{V_1}=\SL_2(q^{1/2})\ \text{ and }\ D=I^{V_2}=\D_{2(q+1)}.
\]
Let $E_i,F_i$ be a hyperbolic pair in $V_i$ for $i=1,2$, and let $\phi_\sharp\in\GaU(V_\sharp)$ such that
\[
\phi_\sharp\colon a_1E_1+b_1F_1+a_2E_2+b_2F_2\mapsto a_1^2E_1+b_1^2F_1+a_2^2E_2+b_2^2F_2
\]
for $a_1,b_1,a_2,b_2\in\bbF_{q^2}$.

For each $z\in J=X\cap Y$, since $z$ normalizes $I$ while $I$ stabilizes $V_1$ and $V_2$, it follows that $I=I^z$ stabilizes $V_1^z$ and $V_2^z$.
Hence $I$ stabilizes $V_i^z\cap V_j$ for $i,j\in\{1,2\}$.
Since $I^{V_1}=\SL_2(q^{1/2})$ and $I^{V_2}=\D_{2(q+1)}$, we conclude that $V_1^z=V_1$ and $V_2^z=V_2$. Thus $J$ stabilizes $V_1$ and $V_2$. Let
\[
X_0=X_{V_1,V_2},\quad Y_0=Y_{V_1,V_2},\quad Z_0=Z_{V_1,V_2}.
\]
Then $X\cap Y\leqslant X_0\cap Y_0$, and $I^{V_i}\leqslant J^{V_i}\leqslant X_0^{V_i}\leqslant Z_0^{V_i}\leqslant\GaO_4^+(q)$ for $i\in\{1,2\}$.
Note that $I^{V_1}=\SL_2(q^{1/2})$, $I^{V_2}=\D_{2(q+1)}$, $X_0=X_{V_1,V_2}$, and $X=\Omega_8^-(q^{1/2}).f$ is a $\mathcal{C}_5$-subgroup of $Z$.
We conclude $(X_0^{V_1})^{(\infty)}=\Omega_4^+(q^{1/2})$ and $(X_0^{V_2})^{(\infty)}=\Omega_4^-(q^{1/2})$.

Suppose that $X\cap Y\neq X\cap Y\cap L$. Then there exists $g\in X\cap Y\cap L$ such that $g\phi_\sharp^f\in X\cap Y$.
It follows that $(g\phi_\sharp^f)|_{V_1}\in\GaO(V_1)\setminus\GO(V_1)$. However, as
\[
(X_0^{V_1})^{(\infty)}=\Omega_4^+(q^{1/2}),\quad Y_0^{V_1}=\GaU(V_1)=\GaU_2(q),\quad Z_0^{V_1}\leqslant\GaO_4^+(q),
\]
Lemma~\ref{LemOmegaPlus32} asserts $X_0^{V_1}\cap Y_0^{V_1}<\Soc(Z_0^{V_1})$, which implies that
\[
(g\phi_\sharp^f)|_{V_1}\in(X\cap Y)^{V_1}\leqslant(X_0\cap Y_0)^{V_1}<\Soc(Z_0^{V_1})<\GO(V_1),
\]
a contradiction. Thus $X\cap Y=X\cap Y\cap L=I=\SL_2(q^{1/2})\times\D_{2(q+1)}$, and so
\[
\frac{|X|}{|X\cap Y|}=\frac{|\Omega_8^-(q^{1/2}).f|}{|\SL_2(q^{1/2})\times\D_{2(q+1)}|}=\frac{q^6(q^3-1)(q^4-1)}{2(q+1)}=\frac{|\Omega_8^+(q).f|}{|\GU_4(q).(2f)|}=\frac{|Z|}{|Y|}
\]
as $q\in\{4,16\}$. Therefore, $Z=XY$.
\end{proof}

Based on Lemma~\ref{LemOmegaPlus33}, we can establish the following result, which gives the pairs $(X,Y)$ in rows~23 and~24 of Table~\ref{TabOmegaPlus}.

\begin{lemma}\label{LemOmegaPlus23}
Let $Z=\Omega(V){:}\langle\phi\rangle=\Omega_8^+(q){:}f$ with $m=4$ and $q\in\{4,16\}$, let $X=\Omega_8^-(q^{1/2}).f$ be a $\mathcal{C}_5$-subgroup of $Z$, and let $Y=\SU_4(q).(2f)<\GU_4(q).(2f)$ be a $\mathcal{C}_3$-subgroup of $Z$. Then $Z=XY$ with $X\cap Y=\SL_2(q^{1/2})\times2$.
\end{lemma}

\begin{proof}
Let $L=\Soc(Z)$, and let $M=\GU_4(q).(2f)<Z$ such that $Y<M$. As Lemma~\ref{LemOmegaPlus33} asserts, $X\cap M=X\cap M\cap L=\SL_2(q^{1/2})\times\D_{2(q+1)}$. Hence $X\cap Y=X\cap Y\cap L\leqslant\SL_2(q^{1/2})\times\D_{2(q+1)}$. Computation in \magma~\cite{BCP1997} shows that $X\cap M^{(\infty)}=\SL_2(q^{1/2})$. Since $M^{(\infty)}=\SU_4(q)$ has index $2$ in $Y\cap L=\SU_m(q).2$, the index of $X\cap M^{(\infty)}$ in $X\cap Y\cap L=X\cap Y$ is at most $2$. Moreover, since $Y$ has index $q+1$ in $M$, the index of $X\cap Y$ in $X\cap M$ is at most $q+1$. Then we derive from the subgroup sequence
\[
\SL_2(q^{1/2})=X\cap M^{(\infty)}\leqslant X\cap Y\leqslant X\cap M=\SL_2(q^{1/2})\times\D_{2(q+1)}
\]
that $X\cap Y=\SL_2(q^{1/2})\times2$. As $q\in\{4,16\}$, it follows that
\[
\frac{|X|}{|X\cap Y|}=\frac{|\Omega_8^-(q^{1/2}).f|}{|\SL_2(q^{1/2})\times2|}=\frac{q^6(q^3-1)(q^4-1)}{2}=\frac{|\Omega_8^+(q).f|}{|\SU_4(q).(2f)|}=\frac{|Z|}{|Y|},
\]
and so $Z=XY$.
\end{proof}

Now we construct the pairs $(X,Y)$ in row~25 of Table~\ref{TabOmegaPlus}.

\begin{lemma}\label{LemOmegaPlus34}
Let $Z=\Omega_{12}^+(2)$, let $M=\SU_6(2)<Z$, let $X$ be a subgroup of $M$ isomorphic to $3.\PSU_4(3)$ or $3.\M_{22}$, and let $Y=\N_1[Z]=\Sp_{10}(2)$. Then $Z=XY$ with
\[
X\cap Y=
\begin{cases}
3^4{:}\A_5&\textup{if }X=3.\PSU_4(3)\\
\PSL_2(11)&\textup{if }X=3.\M_{22}.
\end{cases}
\]
\end{lemma}

\begin{proof}
Since $\SU_5(2)$ intersects trivially with the center of $\SU_6(2)$, we conclude from Lemmas~\ref{LemOmegaPlus07} and~\cite[Lemma~5.6]{LWX-Unitary} that
\[
X\cap Y=X\cap(M\cap Y)=X\cap\SU_5(2)=
\begin{cases}
3^4{:}\A_5&\textup{if }X=3.\PSU_4(3)\\
\PSL_2(11)&\textup{if }X=3.\M_{22}.
\end{cases}
\]
Consequently,
\[
\frac{|X|}{|X\cap Y|}=2^5\cdot3^2\cdot7=\frac{|\Omega_{12}^+(2)|}{|\Sp_{10}(2)|}=\frac{|Z|}{|Y|},
\]
and hence $Z=XY$.
\end{proof}

The next lemma presents the pairs $(X,Y)$ in rows~27 and~28 of Table~\ref{TabOmegaPlus}. Recall the definition of $\psi$ in Section~\ref{SecOmegaPlus01}.

\begin{lemma}\label{LemOmegaPlus35}
Let $Z=\Omega(V){:}\langle\phi\rangle=\Omega_{16}^+(q){:}f$ with $m=8$ and $q\in\{2,4\}$, and let $Y=Z_{e_1+f_1}=\GaSp_{14}(q)$.
\begin{enumerate}[{\rm (a)}]
\item If $M=\Omega(V_\sharp){:}\langle\psi\rangle=\Omega_8^+(q^2).(2f)$ and $X$ is a $\calC_9$-subgroup of $M$ isomorphic to $\Sp_6(q^2).(2f)$ or $\Omega_8^-(q).(2f)$, then $Z=XY$ with 
\[
X\cap Y=
\begin{cases}
\G_2(q^2)&\textup{if }X=\Sp_6(q^2).(2f)\\
\G_2(q)&\textup{if }X=\Omega_8^-(q).(2f).
\end{cases}
\]
\item If $M=\Sp_6(q^2).(2f)<Z$ such that $M$ is a $\calC_9$-subgroup of $\Omega(V_\sharp){:}\langle\psi\rangle=\Omega_8^+(q^2).(2f)$ and $X$ is a subgroup of $M$ isomorphic to $\Omega_6^+(q^2).(2f)$, $\Omega_6^-(q^2).(2f)$ or $\Sp_4(q^2).(2f)$, then $Z=XY$ with
\[
X\cap Y=
\begin{cases}
\SL_3(q^2)&\textup{if }X=\Omega_6^+(q^2).(2f)\\
\SU_3(q^2)&\textup{if }X=\Omega_6^-(q^2).(2f)\\
\SL_2(q^2)&\textup{if }X=\Sp_4(q^2).(2f).
\end{cases}
\]
\item If $M=\Sp_8(q).f$ is a $\calC_9$-subgroup of $Z$ and $X=\Sp_4(q^2).(2f)<M$ then $Z=XY$ with $X\cap Y=\Sp_2(q^2)$.
\end{enumerate}
\end{lemma}

\begin{proof}
First assume that $M$ and $X$ are as in part~(a). It is shown in Lemma~\ref{LemOmegaPlus12} that $Z=MY$ with $M\cap Y=\Sp_6(q^2)$. Thus, for each subgroup $X$ of $M$, we have $Z=XY$ if and only if $M=X(M\cap Y)$. If $X=\Sp_6(q^2).(2f)$, then we see from Lemma~\ref{LemOmegaPlus26} that $M=X(M\cap Y)$ and $X\cap(M\cap Y)=\G_2(q^2)$. Similarly, if $X=\Omega_8^-(q).(2f)$, then by Lemma~\ref{LemOmegaPlus28} we have $M=X(M\cap Y)$ and $X\cap(M\cap Y)=\G_2(q)$. Note that $X\cap Y=X\cap(M\cap Y)$. This proves part~(a).

Next assume that $M$ and $X$ are as in part~(b). For each subgroup $X$ of $M$, since Lemma~\ref{LemOmegaPlus35} shows that $Z=MY$ with $M\cap Y=\G_2(q^2)$, we have $Z=XY$ if and only if $M=X(M\cap Y)$. If $X=\Omega_6^\varepsilon(q^2).(2f)$ with $\varepsilon\in\{+,-\}$, then it follows from Lemma~\ref{LemSymplectic07} that $M=X(M\cap Y)$ and $X\cap(M\cap Y)=\SL_3^\varepsilon(q^2)$. Similarly, if $X=\Sp_4(q^2).(2f)$, then by~\cite[Lemma~4.10]{LWX-Linear} and Lemma~\ref{LemSymplectic07} we have $M=X(M\cap Y)$ and $X\cap(M\cap Y)=\SL_2(q^2)$. Hence part~(b) holds.

Now assume that $M$ and $X$ are given in part~(c). From Lemma~\ref{LemOmegaPlus29} we deduce that $Z=MY$ with $M\cap Y=\Sp_6(q).f<\N_2[M]$. By Lemma~\ref{LemSymplectic13} we have $M=X(M\cap Y)$ with $X\cap(M\cap Y)=\Sp_2(q^2)$. Hence $Z=MY=X(M\cap Y)Y=XY$ with $X\cap Y=X\cap(M\cap Y)=\Sp_2(q^2)$, proving part~(c).
\end{proof}

The pairs $(X,Y)$ in rows~29 and~30 of Table~\ref{TabOmegaPlus} will be constructed in the following two lemmas.

\begin{lemma}\label{LemOmegaPlus36}
Let $Z=\Omega(V){:}\langle\phi\rangle=\Omega_{24}^+(q){:}f$ with $m=12$ and $q\in\{2,4\}$, and let $Y=Z_{e_1+f_1}=\GaSp_{22}(q)$.
\begin{enumerate}[{\rm (a)}]
\item If $M=T{:}\langle\rho\rangle=\SL_{12}(q){:}\langle\rho\rangle$ with $\rho\in\{\phi,\phi\gamma\}$ and $X=\G_2(q^2).(2f)<M$ is the group $H$ defined in~\cite[Lemmas~5.6~or~5.7]{LWX-Linear}, then $Z=XY$ with $X\cap Y=\SL_2(q^2)$.
\item If $M=\SU(V_\sharp){:}\langle\xi\rangle=\SU_{12}(q).(2f)$ and $X=\G_2(q^2).(2f)<M$ is the group $H$ defined in~\cite[Lemma~5.9]{LWX-Unitary} or~\cite[Lemma~5.10]{LWX-Unitary}, then $Z=XY$ with $X\cap Y=\SL_2(q^2)$.
\item If $M=\Sp_6(q^2).(2f)<\Sp_{12}(q).f<(\Sp_2(q)\otimes\Sp_{12}(q)).f<Z$ and $X=\G_2(q^2).(2f)<M$ such that $X\cap\Omega(V)=\G_2(q^2).2$, then $Z=XY$ with $X\cap Y=\SL_2(q^2)$.
\item If $M=\SU_6(q^2).(4f)<\Omega(V_\sharp){:}\langle\psi\rangle=\Omega_{12}^+(q^2).(2f)$ and $X=\G_2(q^2).(4f)<\Sp_6(q^2).(4f)<M$, then $Z=XY$ with $X\cap Y=\SL_2(q^2).2$.
\end{enumerate}
\end{lemma}

\begin{proof}
First assume that $M$ and $X$ are given in part~(a). By Lemma~\ref{LemOmegaPlus03} we have $Z=TZ_{e_1+f_1}$ with $T_{e_1+f_1}=T\cap Z_{e_1+f_1}=\SL_{m-1}(q)$. Hence $Z=T\langle\rho\rangle Z_{e_1+f_1}=MY$ with $M\cap Y=T_{e_1+f_1}{:}\langle\rho\rangle$. Then it follows from~\cite[Lemmas~5.6~and~5.7]{LWX-Linear} that $M=X(M\cap Y)$ with $X\cap(M\cap Y)=X\cap(T_{e_1+f_1}\langle\rho\rangle)=\SL_2(q^2)$, and so $Z=XY$ with $X\cap Y=X\cap(M\cap Y)=\SL_2(q^2)$.

Next assume that $M$ and $X$ are given in part~(b). From Lemma~\ref{LemOmegaPlus07} we obtain $M'\cap Y=\SU_{11}(q)$ and thus $M\cap Y=\SU_{11}(q).(2f)$.
If $X$ is the group $H$ defined in \cite[Lemma~5.9]{LWX-Unitary}, then $X<M'$, and so $X\cap Y=X\cap(M'\cap Y)=\SL_2(q^2)$ by~\cite[Lemma~5.9]{LWX-Unitary}.
If $X$ is the group $H$ defined in~\cite[Lemma~5.10]{LWX-Unitary}, then $X\cap Y=X\cap(M\cap Y)=\SL_2(q^2)$ by~\cite[Lemma~5.10]{LWX-Unitary}.
In either case, we have $X\cap Y=\SL_2(q^2)$. As $q\in\{2,4\}$, it follows that
\begin{equation}\label{EqnOmegaPlus16}
\frac{|X|}{|X\cap Y|}=\frac{|\G_2(q^2).(2f)|}{|\SL_2(q^2)|}=2fq^{10}(q^{12}-1)=q^{11}(q^{12}-1)=\frac{|\Omega_{24}^+(q).f|}{|\GaSp_{22}(q)|}=\frac{|Z|}{|Y|}.
\end{equation}
Therefore, $Z=XY$.

Now assume that $M$ and $X$ are as in part~(c). From $X\cap\Omega(V)=\G_2(q^2).2$ we deduce that $M\Omega(V)/\Omega(V)\geqslant X\Omega(V)/\Omega(V)=f$ and thus $M\Omega(V)=Z$. This implies that $M\cap\Omega(V)=\Sp_6(q^2).2$, and so Lemma~\ref{LemOmegaPlus10} gives $M\cap Y=\Sp_4(q^2)<\N_2[M']$. Then by~\cite[Lemma~4.10]{LWX-Linear} we obtain $X\cap Y=X\cap(M\cap Y)=\SL_2(q^2)$. Again, by~\eqref{EqnOmegaPlus16}, this leads to $Z=XY$.

Finally, let $M$ and $X$ be as in part~(d). By Lemma~\ref{LemOmegaPlus45}(d), $Z=MY$ with $M\cap Y=\SU_5(q^2).2$. Moreover,~\cite[Lemma~4.8]{LWX-Unitary} implies that $M=X(M\cap Y)$ with $X\cap(M\cap Y)=\SL_2(q^2).2$. Thus $Z=XY$ with $X\cap Y=\SL_2(q^2).2$.
\end{proof}

\begin{lemma}\label{LemOmegaPlus38}
Let $Z=\Omega(V)=\Omega_{24}^+(2)$ with $m=12$ and $q=2$, and let $Y=\N_1[Z]=\Sp_{22}(2)$. 
\begin{enumerate}[{\rm (a)}]
\item If $M=\SU(V_\sharp)=\SU_{12}(2)$ and $X=3.\Suz$ is a maximal subgroup of $M$, then $Z=XY$ with $X\cap Y=3^5.\PSL_2(11)$.
\item If $M=\Co_1$ is a maximal subgroup of $Z$ and $X$ is a subgroup of $M$ isomorphic to $\Co_1$, $3.\Suz$ or $\G_2(4).2$, then $Z=XY$ with
\[
X\cap Y=
\begin{cases}
\Co_3&\textup{if }X=\Co_1\\
3^5{:}\PSL_2(11)&\textup{if }X=3.\Suz\\
\A_5&\textup{if }X=\G_2(4).2.
\end{cases}
\]
\end{enumerate}
\end{lemma}

\begin{proof}
First assume that $M$ and $X$ are as in part~(a). Since $\SU_{11}(2)$ intersects trivially with the center of $\SU_{12}(2)$, we conclude from Lemma~\ref{LemOmegaPlus07} and~\cite[Lemma~5.8]{LWX-Unitary} that
\[
X\cap Y=X\cap(M\cap Y)=X\cap\SU_{11}(2)=3^5.\PSL_2(11).
\]
Therefore,
\[
\frac{|X|}{|X\cap Y|}=\frac{|3.\Suz|}{|3^5.\PSL_2(11)|}=2^{11}(2^{12}-1)=\frac{|\Omega_{24}^+(2)|}{|\Sp_{22}(2)|}=\frac{|Z|}{|Y|},
\]
and so $Z=XY$.

Next assume that $M$ and $X$ are as in part~(b). It is proved in~\cite[Page~79,~Lemma~B]{LPS1990} that $Z=MY$ with $M\cap Y=\Co_3$. Thus the conclusion for $X=\Co_1$ holds. If $X=3.\Suz$ or $\G_2(4).2$, then it is shown in~\cite[Pages~319--320]{Giudici2006} that $M=X(M\cap Y)$ and
\[
X\cap(M\cap Y)=
\begin{cases}
3^5{:}\PSL_2(11)&\textup{if }X=3.\Suz\\
\A_5&\textup{if }X=\G_2(4).2.
\end{cases}
\]
Hence $Z=MY=X(M\cap Y)Y=XY$, and $X\cap Y=X\cap(M\cap Y)$ is as described in part~(b).
\end{proof}





Finally, we construct the pairs $(X,Y)$ in rows~31 and~32 of Table~\ref{TabOmegaPlus} in the following lemma.

\begin{lemma}\label{LemOmegaPlus09}
Let $Z=\Omega(V){:}\langle\phi\rangle=\Omega_{32}^+(q){:}f$ with $m=16$ and $q\in\{2,4\}$, let $M=\Omega(V_\sharp){:}\langle\psi\rangle=\Omega_{16}(q^2).(2f)$, let $X=\Sp_8(q^2).(2f)$ be a $\calC_9$-subgroup of $M$, and let $Y=Z_{e_1+f_1}=\GaSp_{30}(q)$. Then $Z=XY$ with $X\cap Y=\Sp_6(q^2)$.
\end{lemma}

\begin{proof}
By Lemma~\ref{LemOmegaPlus12} we have $Z=MY$ with $M\cap Y=\Sp_{14}(q^2)<M'$. Note that $M=XM'$ with $X\cap M'=\Sp_8(q^2)$. We derive from Lemma~\ref{LemOmegaPlus29} that $M'=(X\cap M')(M\cap Y)$ with $(X\cap M')\cap(M\cap Y)=\Sp_6(q^2)$. It follows that $M=X((X\cap M')(M\cap Y))=X(M\cap Y)$ and 
\[
X\cap Y=X\cap(M\cap Y)=X\cap(M'\cap(M\cap Y))=(X\cap M')\cap(M\cap Y)=\Sp_6(q^2).
\] 
Hence $Z=MY=X(M\cap Y)Y=XY$, and the lemma holds.
\end{proof}

\section{Proof of Theorem~\ref{ThmOmegaPlus}}

Let $G$ be an almost simple group with socle $L=\POm_{2m}^+(q)$, and let $H$ and $K$ be subgroups of $G$ not containing $L$ such that both $H$ and $K$ have a unique nonsolvable composition factor. For $L=\Omega_8^+(2)$ or $\POm_8^+(3)$, computation in \magma~\cite{BCP1997} shows that $G=HK$ if and only if $G/L=(HL/L)(KL/L)$ and one of parts~(a)--(c) of Theorem~\ref{ThmOmegaPlus} holds. Thus we assume that $(m,q)\neq(4,2)$ or $(4,3)$ for the rest of this section.

Recall that if $q$ is odd, then $L$ has two orbits on the set of nonsingular $1$-spaces, which are fused by certain similarity of the orthogonal space.
Also recall the notation introduced in Section~\ref{SecOmegaPlus01}.

\begin{lemma}\label{LemOmegaPlus43}
Suppose that $H^{(\infty)}=\lefthat(P.S)\leqslant\Pa_m[L]$ with $P\leqslant R$ and $S\leqslant T$, and $K^{(\infty)}=\Omega_{2m-1}(q)\leqslant\N_1[G]$. Then $G=HK$ if and only if $G/L=(HL/L)(KL/L)$, $HK\supseteq R$, $S=\SL_a(q^b)$ ($m=ab$), $\Sp_a(q^b)'$ ($m=ab$), $\G_2(q^b)'$ ($m=6b$, $q$ even) or $\SL_2(13)$ ($m=6$, $q=3$), and $S$ is defined over $\bbF_{q^b}$ ($b$ is taken to be $1$ if $S=\SL_2(13)$).
\end{lemma}

\begin{proof}
Suppose $G=HK$. Then we obviously have $G/L=(HL/L)(KL/L)$ and $HK\supseteq R$. Moreover, we have $G=H\N_1[G]$ as $K\leqslant\N_1[G]$. Thus $H$ is transitive on a $G$-orbit on the set of nonsingular $1$-spaces. As $(m,q)\neq(4,2)$ or $(4,3)$, it follows from~\cite[Lemmas~4.5~and~3.1]{LPS2010} that $S=\SL_a(q^{2b})$ ($m=ab$), $\Sp_a(q^b)'$ ($m=ab$), $\G_2(q^b)'$ ($m=6b$, $q$ even) or $\SL_2(13)$ ($m=6$, $q=3$), and $S$ is defined over $\bbF_{q^b}$ ($b$ is taken to be $1$ if $S=\SL_2(13)$).

Conversely, suppose that $G/L=(HL/L)(KL/L)$, $HK\supseteq R$, and $S$ is as above.
Let $M=\overline{R{:}T}\cong R.\overline{T}$. Then by Lemmas~\ref{LemOmegaPlus01} and~\ref{LemOmegaPlus02} we have $MK\supseteq MK^{(\infty)}\supseteq L$ with
\[
(M\cap K)R/R\geqslant q^{m-1}{:}\SL_{m-1}(q).
\]
Since $\overline{S}\lesssim HR/R$, we then conclude from~\cite[Lemmas~4.1,~4.2,~4.4,~5.5]{LWX-Linear} that
\[
\overline{T}\subseteq(HR/R)((M\cap K)R/R).
\]
This together with the condition $HK\supseteq R$ implies that $HK\supseteq R.\overline{T}=M$ and so $HK\supseteq MK\supseteq L$,
Thus by Lemma~\ref{LemXia01} we obtain $G=HK$ as $G/L=(HL/L)(KL/L)$.
\end{proof}

Results in Sections~\ref{SecOmegaPlus02}--\ref{SecOmegaPlus03} together with Lemmas~\ref{LemXia02} and~\ref{LemOmegaPlus43} imply that if $G/L=(HL/L)(KL/L)$ and $(H,K)$ is described in parts~(a)--(c) of Theorem~\ref{ThmOmegaPlus}, then $G=HK$.

In what follows, suppose that $G=HK$. By~\cite{LWX}, interchanging $H$ and $K$ if necessary, the triple $(L,H^{(\infty)},K^{(\infty)})$ lies in Table~\ref{TabInftyOmegaPlus}, and in particular,
\begin{equation}\label{EqnOmegaPlus17}
K^{(\infty)}=\Omega_{2m-1}(q),\ \ \Omega_{2m-2}^-(q)\ \text{ or }\ \lefthat\SU_m(q).
\end{equation}
Let $A$ and $B$ be $\max^-$ subgroups of $G$ containing $H$ and $K$, respectively. By~\cite[Theorem~A]{LPS1990}\footnote{See~\cite[Remark~1.5]{GGP} for a correction.} and~\cite{LPS1996}, the triple $(L,A\cap L,B\cap L)$ lies in Table~\ref{TabMaxOmegaPlus2} or~\ref{TabMaxOmegaPlus1} (with $A$ and $B$ possibly interchanged). We complete the proof by discussing the three cases in~\eqref{EqnOmegaPlus17} for $K^{(\infty)}$.

\begin{table}[htbp]
\captionsetup{justification=centering}
\caption{$(L,H^{(\infty)},K^{(\infty)})$ for orthogonal groups of plus type, $L\neq\Omega_8^+(2)$, $\POm_8^+(3)$}\label{TabInftyOmegaPlus}
\begin{tabular}{|l|l|l|l|l|l|}
\hline
 & $L$ & $H^{(\infty)}$ & $K^{(\infty)}$ & Conditions\\
\hline
1 & $\POm_{2m}^+(q)$ & $\lefthat(P.\SL_a(q^b))$ ($m=ab$), & $\Omega_{2m-1}(q)$ & $P\leqslant q^{m(m-1)/2}$\\
 & & $\lefthat(P.\Sp_a(q^b)')$ ($m=ab$), & & \\
 & & $\lefthat(P.\G_2(q^b)')$ ($m=6b$, $q$ even), & & \\
2 & $\POm_{2m}^+(q)$ & $\lefthat\SU_m(q)$ ($m$ even), $\PSp_m(q) $& $\Omega_{2m-1}(q)$ & \\
\hline
3 & $\Omega_{2m}^+(2)$ & $\SU_{m/2}(4)$, $\Omega_m^+(4)$, $\SU_{m/4}(16)$, $\Omega_{m/2}^+(16)$ & $\Sp_{2m-2}(2)$ & \\
4 & $\Omega_{2m}^+(4)$ & $\SU_{m/2}(16)$, $\Omega_m^+(16)$ & $\Sp_{2m-2}(4)$ & \\
5 & $\POm_8^+(q)$ & $\Omega_8^-(q^{1/2})$ ($q$ square), $\Omega_7(q)$ &  $\Omega_7(q)$ & \\
6 & $\Omega_{12}^+(q)$ & $\G_2(q)'$ & $\Sp_{10}(q)$ & $q$ even\\
7 & $\POm_{16}^+(q)$ & $\Omega_9(q)$ & $\Omega_{15}(q)$ & \\
\hline
8 & $\POm_{2m}^+(q)$ & $\lefthat(q^{m(m-1)/2}{:}\SL_m(q))$ & $\Omega_{2m-2}^-(q)$ & \\
9 & $\Omega_{2m}^+(2)$ & $\SL_m(2)$ & $\Omega_{2m-2}^-(2)$ & \\
10 & $\Omega_{2m}^+(4)$ & $\SL_m(4)$ & $\Omega_{2m-2}^-(4)$ & \\
\hline
11 & $\POm_{2m}^+(q)$ & $q^{2m-2}{:}\Omega_{2m-2}^+(q)$ & $\lefthat\SU_m(q)$ & $m$ even\\
12 & $\Omega_{2m}^+(2)$ & $\Omega^+_{2m-2}(2)$ & $\SU_m(2)$ & $m$ even\\
13 & $\Omega_{2m}^+(4)$ & $\Omega^+_{2m-2}(4)$ & $\SU_m(4)$ & $m$ even\\
\hline
14 & $\Omega_8^+(4)$ & $\Omega_8^-(2)$ & $\SU_4(4)$ & \\
15 & $\Omega_8^+(16)$ & $\Omega_8^-(4)$ & $\SU_4(16)$ & \\
16 & $\Omega_{12}^+(2)$ & $3.\PSU_4(3)$, $3.\M_{22}$ & $\Sp_{10}(2)$ & \\
17 & $\POm_{12}^+(3)$ & $P.\PSL_2(13)$ & $\Omega_{11}(3)$ & $P\leqslant3^{15}$\\
18 & $\Omega_{16}^+(2)$ & $\Omega_8^-(2)$, $\Sp_6(4)$ & $\Sp_{14}(2)$ & \\
19 & $\Omega_{16}^+(4)$ &  $\Omega_8^-(4)$, $\Sp_6(16)$ & $\Sp_{14}(4)$ & \\
20 & $\Omega_{24}^+(2)$ & $\G_2(4)$, $3.\Suz$, $\Co_1$ & $\Sp_{22}(2)$ & \\
21 & $\Omega_{24}^+(4)$ & $\G_2(16)$ & $\Sp_{22}(4)$ & \\
22 & $\Omega_{32}^+(2)$ & $\Omega_8^-(4)$, $\Sp_6(16)$ & $\Sp_{30}(2)$ & \\
23 & $\Omega_{48}^+(2)$ & $\G_2(16)$ & $\Sp_{46}(2)$ & \\
\hline
\end{tabular}
\vspace{3mm}
\end{table}

\begin{table}[htbp]
\caption{Maximal factorizations with socle $L=\POm_8^+(q)$, $q\geqslant4$}\label{TabMaxOmegaPlus2}
\centering
\begin{tabular}{|l|l|l|l|}
\hline
Row & $A\cap L$ & $B\cap L$ & Remark\\
\hline
1 & $\Omega_7(q)$ & $\Omega_7(q)$ & $A=B^\tau$ for some triality $\tau$ \\
2 & $\Pa_1$, $\Pa_3$, $\Pa_4$ & $\Omega_7(q)$ & \\
3 & $\lefthat((q+1)/d\times\Omega_6^-(q)).2^d$ & $\Omega_7(q)$ & $d=(2,q-1)$, $A$ in $\mathcal{C}_1$ or $\mathcal{C}_3$ \\
4 & $\lefthat((q-1)/d\times\Omega_6^+(q)).2^d$ & $\Omega_7(q)$ & $d=(2,q-1)$, $A$ in $\mathcal{C}_1$ or $\mathcal{C}_2$, \\
 & & & $q>2$ \\
5 & $(\PSp_2(q)\otimes\PSp_4(q)).2$ & $\Omega_7(q)$ & $q$ odd, $A$ in $\mathcal{C}_1$ or $\mathcal{C}_4$ \\
6 & $\Omega_8^-(q^{1/2})$ & $\Omega_7(q)$ & $q$ square, $A$ in $\mathcal{C}_5$ or $\mathcal{S}$ \\
7 & $\Pa_1$, $\Pa_3$, $\Pa_4$ & $\lefthat((q+1)/d\times\Omega_6^-(q)).2^d$ & $d=(2,q-1)$, $B$ in $\mathcal{C}_1$ or $\mathcal{C}_3$ \\
\hline
8 & $(\SL_2(16)\times\SL_2(16)).2^2$ & $\Omega_7(4)$ & $q=4$ \\
9 & $(3\times\Omega_6^+(4)).2$, $\Omega_8^-(2)$ & $(5\times\Omega_6^-(4)).2$ & $q=4$ \\
10 & $\Omega_8^-(4)$ & $(17\times\Omega_6^-(16)).2$ & $q=16$ \\
\hline
\end{tabular}
\end{table}

\begin{table}[htbp]
\caption{Maximal factorizations with socle $L=\POm_{2m}^+(q)$, $m\geqslant5$}\label{TabMaxOmegaPlus1}
\centering
\begin{tabular}{|l|l|l|l|}
\hline
Row & $A\cap L$ & $B\cap L$ & Remark\\
\hline
1 & $\Pa_m$, $\Pa_{m-1}$ & $\N_1$ & \\
2 & $\lefthat\GU_m(q).2$ & $\N_1$ & $m$ even \\
3 & $(\PSp_2(q)\otimes\PSp_m(q)).a$ & $\N_1$ & $m$ even, $q>2$, $a=\gcd(2,m/2,q-1)$ \\
4 & $\Pa_m$, $\Pa_{m-1}$ & $\N_2^-$ & \\
5 & $\lefthat\GU_m(q).2$ & $\Pa_1$ & $m$ even \\
6 & $\lefthat\GL_m(q).2$ & $\N_1$ & \\
7 & $\Omega_m^+(4).2^2$ & $\N_1$ & $m$ even, $q=2$ \\
8 & $\Omega_m^+(16).2^2$ & $\N_1$ & $m$ even, $q=4$ \\
9 & $\lefthat\GL_m(2).2$ & $\N_2^-$ & $q=2$ \\
10 & $\lefthat\GL_m(4).2$ & $\N_2^-$ & $q=4$ \\
11 & $\lefthat\GU_m(4).2$ & $\N_2^+$ & $m$ even, $q=4$ \\
12 & $\Omega_9(q).a$ & $\N_1$ & $m=8$, $a\leqslant2$ \\
13 & $\Co_1$ & $\N_1$ & $m=12$, $q=2$ \\
\hline
\end{tabular}
\end{table}

\begin{lemma}
Suppose $K^{(\infty)}=\Omega_{2m-1}(q)$. Then either $(H,K)$ tightly contains $(X^\alpha,Y^\alpha)$ for some $(X,Y)$ in rows~\emph{1--6, 22 or 25--32} of Table~$\ref{TabOmegaPlus}$ and $\alpha\in\Aut(L)$, or $(H,K)$ satisfies part~\emph{(b)} or~\emph{(c)} of Theorem~$\ref{ThmOmegaPlus}$.
\end{lemma}

\begin{proof}
Since $K^{(\infty)}=\Omega_{2m-1}(q)$, we see from Tables~\ref{TabMaxOmegaPlus2} and~\ref{TabMaxOmegaPlus1} that, replacing $H$ and $K$ by their images under some automorphism of $L$ if necessary, $B\cap L=\N_1$.

\textbf{Claim}: $H\nleqslant\mathrm{O}(V_\sharp){:}\langle\psi^2\rangle$ if $q$ is even. Suppose for a contradiction that $H\leqslant M:=\mathrm{O}(V_\sharp){:}\langle\psi^2\rangle$ with $q$ even. Then $G=HB=MB$. Since $|G|_2/|B|_2\geqslant|L|_2/|B\cap L|_2=|\Omega_{2m}^+(q)|_2/|\Sp_{2m-2}(q)|_2=q^{m-1}$ and
\[
M\cap B\geqslant\mathrm{O}(V_\sharp)\cap B=\mathrm{O}(V_\sharp)\cap\N_1[G]\geqslant\N_1[\mathrm{O}(V_\sharp)]=\N_1[\mathrm{O}_m^+(q^2)]=\Sp_{m-2}(q^2)\times2,
\]
it follows that
\[
q^{m-1}\leqslant\frac{|G|_2}{|B|_2}=\frac{|M|_2}{|M\cap B|_2}=\frac{|\mathrm{O}_m^+(q^2).f|_2}{|M\cap B|_2}\leqslant\frac{|\mathrm{O}_m^+(q^2).f|_2}{|\Sp_{m-2}(q^2)\times2|_2}=q^{m-2}f,
\]
a contradiction. This proves the claim.

\textsc{Case~1}: $(A\cap L,B\cap L)$ lies in row~1 of Table~\ref{TabMaxOmegaPlus2}. In this case, Lemma~\ref{LemOmegaPlus42} implies that either $(H,K)$ tightly contains $(X^\alpha,Y^\alpha)$ for some $(X,Y)$ in row~4 or~22 Table~\ref{TabOmegaPlus} and $\alpha\in\Aut(L)$, or $(H,K)$ satisfies part~(c) of Theorem~\ref{ThmOmegaPlus}.

\textsc{Case~2}: $(A\cap L,B\cap L)$ lies in row~2 of Table~\ref{TabMaxOmegaPlus2} or row~1 of Table~\ref{TabMaxOmegaPlus1}. Applying an involutory graph automorphism of $L$ if necessary, we may assume $A\cap L=\Pa_m$. Then Lemma~\ref{LemOmegaPlus43} shows that $(H,K)$ either tightly contains some $(X,Y)$ in rows~1 of Table~\ref{TabOmegaPlus} or satisfies part~(b) of Theorem~\ref{ThmOmegaPlus}.

\textsc{Case~3}: $(A\cap L,B\cap L)$ lies in row~3 of Table~\ref{TabMaxOmegaPlus2} or row~2 of Table~\ref{TabMaxOmegaPlus1}. Then $m$ is even, and $A\cap L=\lefthat\GU_m(q).2$ is a $\calC_3$-subgroup of $L$. By~\cite[Table~3.5.E]{KL1990}, applying an involutory graph automorphism of $L$ if necessary, we may assume $A\cap L=\lefthat\GU(V_\sharp).2$. Note that Lemma~\ref{LemOmegaPlus07} implies $(A\cap B)^{(\infty)}=\SU_{m-1}(q)$. Thus, by considering the factorization $A=H(A\cap B)$ (and its quotient modulo $\Rad(A)$), we derive from~\cite[Theorem~1.2]{LWX-Unitary} that either $H^{(\infty)}\leqslant\Pa_{m/2}[\lefthat\SU_m(q)]$ or $H^{(\infty)}$ is one of:
\begin{align}
&\lefthat\SU_m(q),\quad \G_2(q)\text{ with $m=6$ and $q$ even},\quad 3.\PSU_4(3)\text{ with $(m,q)=(6,2)$},\label{EqnOmega07}\\
&\PSp_{2m}(q),\quad 3.\M_{22}\text{ with $(m,q)=(6,2)$},\quad 3.\Suz\text{ with $(m,q)=(12,2)$},\\
&\SL_{m/2}(q^2)\text{ with $q\in\{2,4\}$ and $HA^{(\infty)}\geqslant\SU_m(q).(2f)$},\\
&\Sp_{m/2}(q^2)\text{ with $q\in\{2,4\}$ and $HA^{(\infty)}\geqslant\SU_m(q).(2f)$},\\
&\G_2(q^2)\text{ with $m=12$, $q\in\{2,4\}$ and $HA^{(\infty)}\geqslant\SU_m(q).(2f)$}\label{EqnOmega10}.
\end{align}
If $H^{(\infty)}\leqslant\Pa_{m/2}[\lefthat\SU_m(q)]$, then $H^{(\infty)}\leqslant\Pa_m[L]$ and so Lemma~\ref{LemOmegaPlus43} leads to part~(b) of Theorem~\ref{ThmOmegaPlus}. If $H^{(\infty)}$ lies in~\eqref{EqnOmega07}--\eqref{EqnOmega10}, then $(H,K)$ tightly contains some $(X,Y)$ in rows~1--3,~5,~25,~29 or~30 of Table~\ref{TabOmegaPlus}.

\textsc{Case~4}: $(A\cap L,B\cap L)$ lies in row~4 of Table~\ref{TabMaxOmegaPlus2} or row~6 of Table~\ref{TabMaxOmegaPlus1}. Then $A\cap L=\lefthat\GL_m(q).2$ is a $\calC_2$-subgroup of $L$. By~\cite[Table~3.5.E]{KL1990}, applying an involutory graph automorphism of $L$ if necessary, we may assume $A^{(\infty)}=\overline{T}$. Since Lemma~\ref{LemOmegaPlus03} gives $(A\cap B)^{(\infty)}=\SL_{m-1}(q)$, by considering the factorization $A=H(A\cap B)$ (and its quotient modulo $\Rad(A)$), we derive from~\cite[Theorem~1.2]{LWX-Linear} that $H^{(\infty)}$ is one of:
\begin{align}
&\lefthat\SL_m(q),\quad\PSp_{2m}(q),\quad \G_2(q)\text{ with $m=6$ and $q$ even},\label{EqnOmegaPlus11}\\
&\SL_{m/2}(q^2)\text{ with $q\in\{2,4\}$},\quad\Sp_{m/2}(q^2)\text{ with $q\in\{2,4\}$},\label{EqnOmegaPlus12}\\
&\G_2(q^2)\text{ with $m=12$ and $q\in\{2,4\}$}.\label{EqnOmegaPlus13}
\end{align}
If $H^{(\infty)}$ lies in~\eqref{EqnOmegaPlus11}, then $(H,K)$ tightly contains some $(X,Y)$ in row~1 or~5 of of Table~\ref{TabOmegaPlus}. If $H^{(\infty)}$ lies in~\eqref{EqnOmegaPlus12}--\eqref{EqnOmegaPlus13}, then our Claim implies that $(H,K)$ tightly contains some $(X,Y)$ in rows~2--4 or~29--30 of Table~\ref{TabOmegaPlus}.

\textsc{Case~5}: $(A\cap L,B\cap L)$ lies in row~5 of Table~\ref{TabMaxOmegaPlus2} or row~3 of Table~\ref{TabMaxOmegaPlus1}. In this case,
$q\geqslant3$ and $A^{(\infty)}=A_1\times A_2$ with $A_1=\PSp_2(q)^{(\infty)}$ (this is $1$ if $q=3$) and $A_2=\PSp_m(q)$.
By Lemma~\ref{LemOmegaPlus08} we have $(A_2\cap B)^{(\infty)}=\PSp_{m-2}(q)$. Note that $A/\Cen_A(A_2)$ is an almost simple group with socle $\PSp_m(q)$. Then considering the factorization $A=H(A\cap B)$ (and its quotient modulo $\Cen_A(A_2)$), we conclude from~\cite[Theorem~5.1]{LWX} that $H^{(\infty)}$ is one of:
\begin{align}
&\PSp_m(q),\quad\G_2(q)\text{ with $m=6$ and $q$ even},\label{EqnOmegaPlus10}\\
&\Sp_{m/2}(16)\text{ with $q=4$},\quad\G_2(16)\text{ with $m=12$ and $q=4$}\label{EqnOmegaPlus09}.
\end{align}
If $H^{(\infty)}$ lies in~\eqref{EqnOmegaPlus10}, then $(H,K)$ tightly contains some $(X,Y)$ in row~1 or~5 of Table~\ref{TabOmegaPlus}.
If $H^{(\infty)}$ lies in~\eqref{EqnOmegaPlus09}, then we see from~\cite[Theorem~A]{LPS1990} that $A/\Cen_A(A_2)=\GaSp_m(4)$ and $H\Cen_A(A_2)/\Cen_A(A_2)\nleqslant\Soc(A/\Cen_A(A_2))$, whence $(H,K)$ tightly contains some $(X,Y)$ in row~3 or~30 of Table~\ref{TabOmegaPlus}.

\textsc{Case~6}: $(A\cap L,B\cap L)$ lies in row~6 of Table~\ref{TabMaxOmegaPlus2}. Then $A^{(\infty)}=\Omega_8^-(q^{1/2})$, and Lemma~\ref{LemOmegaPlus28} shows that $(A\cap B)^{(\infty)}=\G_2(q^{1/2})$. Since $A=H(A\cap B)$, we derive from~\cite[Theorem~1.2]{LWX-OmegaMinus} that $H^{(\infty)}=A^{(\infty)}$. Thus $(H,K)$ tightly contains $(X,Y)=(\Omega_8^-(q^{1/2}),\Omega_7(q))$ in row~4 of Table~\ref{TabOmegaPlus}.

\textsc{Case~7}: $(A\cap L,B\cap L)$ lies in row~12 of Table~\ref{TabMaxOmegaPlus1}. In this case, $A^{(\infty)}=\Omega_9(q)$, and Lemma~\ref{LemOmegaPlus29} shows that $(A\cap B)^{(\infty)}=\Omega_7(q)<\N_2[A^{(\infty)}]$. Since $A=H(A\cap B)$, it follows from~\cite[Theorem~1.2]{LWX-Omega} and~\cite[Theorem~A]{LPS1990} that either $H^{(\infty)}=A^{(\infty)}$, or $H\leqslant\Sp_4(q^2).(2f)$ with $q\in\{2,4\}$ and $A=\GaSp_8(q)$. For the former, $(H,K)$ tightly contains $(X,Y)=(\Omega_9(q),\Omega_{15}(q))$ in row~6 of Table~\ref{TabOmegaPlus}. Now assume that $H\leqslant\Sp_4(q^2).(2f)$ with $q\in\{2,4\}$ and $A=\GaSp_8(q)$. Then since $A=H(A\cap B)$ with $(A\cap B)^{(\infty)}=\Omega_7(q)<\N_2[A^{(\infty)}]$, computation in \magma~\cite{BCP1997} shows that $H=\Sp_4(q^2).(2f)$, as in row~27 or~38 of Table~\ref{TabOmegaPlus}.

\textsc{Case~8}: $(A\cap L,B\cap L)$ lies in row~13 of Table~\ref{TabMaxOmegaPlus1}. Then $A^{(\infty)}=\Co_1$, and Lemma~\ref{LemOmegaPlus38}(b) shows that $(A\cap B)^{(\infty)}=\Co_3$. Since $A=H(A\cap B)$, we derive from~\cite{Giudici2006} that computation in \magma~\cite{BCP1997} shows that $(H,K)$ tightly contains some $(X,Y)$ in row~29 of Table~\ref{TabOmegaPlus}.

\textsc{Case~9}: $(A\cap L,B\cap L)$ lies in row~8 of Table~\ref{TabMaxOmegaPlus2} or rows~7--8 of Table~\ref{TabMaxOmegaPlus1}. Then $q\in\{2,4\}$, $m$ is even, and $A\cap L=\Omega_m^+(q^2).2^2$ is a $\calC_3$-subgroup of $L$. By~\cite[Table~3.5.E]{KL1990}, applying an involutory graph automorphism of $L$ if necessary, we may assume $A^{(\infty)}=\Omega(V_\sharp)$. If $m=4$ then $(H,K)$ tightly contains some $(X,Y)$ in row~3 or~22 of Table~\ref{TabOmegaPlus}.
Thus we assume $m\geqslant6$.

By Lemma~\ref{LemOmegaPlus12} we have $(A\cap B)^{(\infty)}=\Sp_{m-2}(q^2)=\N_1[A^{(\infty)}]$ and $A\cap B\ngeqslant\GaSp_{m-2}(q^2)$. Note that $A$ is almost simple with socle $A^{(\infty)}=\Omega(V_\sharp)=\Omega_m^+(q^2)$. Appealing to~\cite[Theorem~A]{LPS1990}, we deduce from the factorization $A=H(A\cap B)$ that either $H\geqslant A^{(\infty)}$, or $H\cap A^{(\infty)}$ is contained in one of the groups:
\begin{align}
&\Pa_{m/2-1}[A^{(\infty)}],\quad\Pa_{m/2}[A^{(\infty)}],\quad\GL_{m/2}(q^2).2,\quad\Sp_2(q^2)\otimes\Sp_{m/2}(q^2),\label{EqnOmegaPlus14}\\
&\Omega_{m/2}^+(q^4).2^2\text{ with $q=2$},\quad\GU_{m/2}(q^2).2\text{ with $m/2$ even},\\
&\Sp_6(q^2)\text{ with $m=8$},\quad\Omega_8^-(q)\text{ with $m=8$},\quad\Sp_8(q^2)\text{ with $m=16$}.
\end{align}
Here $\Sp_6(q^2)$, $\Omega_8^-(q)$ and $\Sp_8(q^2)$ are $\calC_9$-subgroups of $A^{(\infty)}$.
If $H\geqslant A^{(\infty)}=\Omega(V_\sharp)$, then since $H\nleqslant\mathrm{O}(V_\sharp){:}\langle\psi^2\rangle$, it follows that $H$ contains (up to conjugate in $G$) either $\Omega(V_\sharp){:}\langle\psi\rangle$ or $\Omega(V_\sharp){:}\langle\psi r'_{E_1+F_1}\rangle$, and so $(H,K)$ tightly contains $(X,Y)=(\Omega_m^+(q^2).(2f),\GaSp_{2m-2}(q))$ in row~2 or~3 of Table~\ref{TabOmegaPlus}.
If $H\cap A^{(\infty)}$ is contained in one of the groups in~\eqref{EqnOmegaPlus14}, then $H\cap L$ is contained in $\Pa_{m-1}$, $\Pa_m$, $\GU_m(q).2$, $\GL_m(q).2$ or $\Sp_2(q)\otimes\Sp_m(q)$, which is dealt with in the previous cases.

Suppose that $H\cap A^{(\infty)}$ is contained in $\Omega_{m/2}^+(q^4).2^2$ with $q=2$. From Lemma~\ref{LemOmegaPlus12} we see that $(\Omega(V_\sharp){:}\langle\psi\rangle)\cap B$ and $(\Omega(V_\sharp){:}\langle\psi r'_{E_1+F_1}\rangle)\cap B$ are both contained in $\Omega(V_\sharp)$. Then since
\[
A\leqslant\Omega(V_\sharp){:}\langle\psi,r'_{E_1+F_1}\rangle=(\Omega(V_\sharp){:}\langle\psi\rangle)\cup(\Omega(V_\sharp){:}\langle\psi r'_{E_1+F_1}\rangle)\cup(\Omega(V_\sharp){:}\langle r'_{E_1+F_1}\rangle),
\]
it follows that $A\cap B\leqslant \Omega(V_\sharp){:}\langle r'_{E_1+F_1}\rangle=\mathrm{O}(V_\sharp)$. Hence the factorization $A=H(A\cap B)$ yields $A\cap\mathrm{O}(V_\sharp)=(H\cap\mathrm{O}(V_\sharp))(A\cap B)$. However, our Claim applied to the factorization $A\cap\mathrm{O}(V_\sharp)=(H\cap\mathrm{O}(V_\sharp))(A\cap B)$ shows that it is impossible.

Assume that $H\cap A^{(\infty)}$ is contained in $\GU_{m/2}(q^2).2$ with $m/2$ even. Write $m=4\ell$. Then $H$ is contained in a subgroup $A_1$ of $A$ such that $A_1^{(\infty)}=\SU_{2\ell}(q^2)$. By Lemma~\ref{LemOmegaPlus45}(d) we have $(A_1\cap B)^{(\infty)}=\SU_{2\ell-1}(q^2)$. Considering the factorization $A_1=H(A_1\cap B)$, we then derive from~\cite[Theorem~1.2]{LWX-Unitary} that either $H^{(\infty)}\leqslant\Pa_\ell[\SU_{2\ell}(q^2)]$, or $H^{(\infty)}\leqslant\SL_\ell(q^4)$ with $q=2$, or $H^{(\infty)}$ is one of the groups:
\begin{equation}\label{EqnOmegaPlus15}
\SU_{2\ell}(q^2),\quad\Sp_{2\ell}(q^2),\quad\G_2(q^2)\text{ with }\ell=3.
\end{equation} 
If $H^{(\infty)}\leqslant\Pa_\ell[\SU_{2\ell}(q^2)]$, then $H^{(\infty)}\leqslant\Pa_m[L]$ and so Lemma~\ref{LemOmegaPlus43} leads to part~(b) of Theorem~\ref{ThmOmegaPlus}. If $H^{(\infty)}\leqslant\SL_\ell(q^4)$ with $q=2$, then $H\cap A^{(\infty)}$ is contained in $\Omega_{m/2}^+(q^4).2^2$, which is impossible by the previous paragraph. If $H^{(\infty)}$ is one of the groups in~\eqref{EqnOmegaPlus15}, then by the Claim, $(H,K)$ tightly contains some pair $(X,Y)$ in rows~2--3 or~29--30 of Table~\ref{TabOmegaPlus}.

Assume that $H\cap A^{(\infty)}$ is contained in $\Sp_6(q^2)$ with $m=8$. Here $H$ is contained in a $\calC_9$-subgroup $A_1$ of $A$ such that $\Sp_6(q^2)\leqslant A_1\leqslant\GaSp_6(q^2)$. Note that $\Pa_1[A_1]\leqslant\Pa_i[A]$ for some $i\in\{3,4\}$ and that $A_1\leqslant\Omega(V_\sharp){:}\langle\psi\rangle$ (see~\cite[Table~8.50]{BHR2013}).
Since $A\cap B\ngeqslant\GaSp_6(q^2)$, Lemma~\ref{LemOmegaPlus42} applied to the factorization $A=H(A\cap B)$ shows that either $H$ is contained in $\Pa_1[A_1]$ or $H^{(\infty)}\in\{\Sp_6(q^2),\Omega_6^+(q^2),\Omega_6^-(q^2),\Sp_4(q^2)\}$.
For the latter, since $H\nleqslant\mathrm{O}(V_\sharp){:}\langle\psi^2\rangle$, we conclude that $(H,K)$ tightly contains some $(X,Y)$ in row~27 or~28 of Table~\ref{TabOmegaPlus}.
If $H\leqslant\Pa_1[A_1]$, then $H\leqslant\Pa_i[A]\leqslant\Pa_j[G]$ for some $j\in\{7,8\}$, which is dealt with in Case~2.

Assume that $H\cap A^{(\infty)}$ is contained in $\Omega_8^-(q)$ with $m=8$. The conclusion in Case~6 applied to the factorization $A=H(A\cap B)$ asserts that $H^{(\infty)}=\Omega_8^-(q)$. Hence we see from~\cite[Table~8.50]{BHR2013} that $H\leqslant\Omega(V_\sharp){:}\langle\psi\rangle$. This together with $H\nleqslant\mathrm{O}(V_\sharp){:}\langle\psi^2\rangle$ implies that $(H,K)$ tightly contains $(X,Y)=(\Omega_8^-(q).(2f),\GaSp_{14}(q))$ in row~27 or~28 of Table~\ref{TabOmegaPlus}.

Assume that $H\cap A^{(\infty)}$ is contained in $\Sp_8(q^2)$ with $m=16$. The conclusion in Case~7 applied to the factorization $A=H(A\cap B)$ shows that either $H^{(\infty)}=\Sp_8(q^2)$, or $H=\Sp_4(16).4$ with $q=2$. The latter implies that $H\leqslant M$ for some subgroup $M$ of $A$ with $M\cap A^{(\infty)}=\Omega_8^+(16).2^2$ and so $G=MB$, which is shown above to be impossible. Hence $H^{(\infty)}=\Sp_8(q^2)$, and according to~\cite[Table~7.8]{Rogers2017}, $H\leqslant\Omega(V_\sharp){:}\langle\psi\rangle$. This together with $H\nleqslant\mathrm{O}(V_\sharp){:}\langle\psi^2\rangle$ implies that $(H,K)$ tightly contains $(X,Y)=(\Sp_8(q^2).(2f),\GaSp_{30}(q))$ in row~31 or~32 of Table~\ref{TabOmegaPlus}.
\end{proof}

Next we handle the case $K^{(\infty)}=\Omega_{2m-2}^-(q)$, which corresponds to the candidates $(H^{(\infty)},K^{(\infty)})$ in rows~8--10 of Table~\ref{TabInftyOmegaPlus}.

\begin{lemma}
Suppose that $(H^{(\infty)},K^{(\infty)})$ lies in rows~\emph{8--10} of Table~$\ref{TabInftyOmegaPlus}$. Then $(H,K)$ tightly contains $(X^\alpha,Y^\alpha)$ for some $(X,Y)$ in rows~\emph{7--10} of Table~$\ref{TabOmegaPlus}$ and $\alpha\in\Aut(L)$.
\end{lemma}

\begin{proof}
Since $(H^{(\infty)},K^{(\infty)})$ lies in rows~8--10 of Table~\ref{TabInftyOmegaPlus}, we have $K^{(\infty)}=\Omega_{2m-2}^-(q)$, and either $H^{(\infty)}=\lefthat(q^{m(m-1)/2}{:}\SL_m(q))$, or $H^{(\infty)}=\SL_m(q)$ with $q\in\{2,4\}$, If $H^{(\infty)}=\lefthat(q^{m(m-1)/2}{:}\SL_m(q))$, then $(H,K)$ tightly contains $(X^\alpha,Y^\alpha)$ for some $(X,Y)$ in row~7 of Table~\ref{TabOmegaPlus} and $\alpha\in\Aut(L)$. For the rest of the proof, assume $H^{(\infty)}=\SL_m(q)$ with $q\in\{2,4\}$.

By Tables~\ref{TabMaxOmegaPlus2} and~\ref{TabMaxOmegaPlus1}, replacing $H$ and $K$ by their images under some automorphism of $L$ if necessary, we have $A\cap L=\Pa_m$ or $\lefthat\GL_m(q).2$. By~\cite[Table~4.5]{CPS1975}, $\Pa_m$ has a unique conjugacy class of subgroups isomorphic to $\SL_m(q)$. Hence it always holds that $H^{(\infty)}<\lefthat\GL_m(q).2$, and so we may assume $A\cap L=\lefthat\GL_m(q).2$. Without loss of generality, assume $H^{(\infty)}=T$. Moreover, $K\leqslant\N_2^-[G]$. Note that $\gamma=r_{e_1+f_1}\cdots r_{e_m+f_m}$ lies in $\Omega(V)$ if and only if $m$ is even. We divide our argument into three cases as follows.

\textsc{Case~1}: $q=2$ and $\gamma\notin H$. Then $H=\SL_m(2)<L$, and so we have $L=H(K\cap L)$. Write $K\cap L=\Omega_{2m-2}^-(2).\calO$ with $|\calO|$ dividing $|\N_2^-|/|\Omega_{2m-2}^-(2)|=6$. From Lemma~\ref{LemOmegaPlus15} we deduce that $L=H\N_2^-$ with $H\cap\N_2^-=\SL_{m-2}(2).3$. Hence
\[
|H\cap(K\cap L)|=\frac{|H||K\cap L|}{|L|}=\frac{|H||\N_2^-|}{|L|(|\N_2^-|/|K\cap L|)}=\frac{|H\cap\N_2^-|}{|\N_2^-|/|K\cap L|}=\frac{|\SL_{m-2}(2).3|}{|\N_2^-|/|K\cap L|},
\]
which means that $H\cap(K\cap L)$ is a subgroup of $\SL_{m-2}(2).3$ of index $|\N_2^-|/|K\cap L|=6/|\calO|$. As a consequence, $|\calO|=2$ or $6$, and so $K\cap L\geqslant\Omega_{2m-2}^-(2).2$. Since all involutions in $\N_2^-/\Omega_{2m-2}^-(2)\cong\mathrm{O}_2^-(2)$ are conjugate, it follows that (up to conjugation in $L$) the pair $(H,K)$ tightly contains the pair $(X,Y)=(\SL_m(2),\Omega_{2m-2}^-(2).2)$ in row~9 of Table~\ref{TabOmegaPlus}.

\textsc{Case~2}: $q=2$ and $\gamma\in H$. Then $H=T{:}\langle\gamma\rangle=\SL_m(2).2$. If $m$ is even, then $(H,K)$ tightly contains the pair $(X,Y)=(\SL_m(2).2,\Omega_{2m-2}^-(2))$ in row~8 of Table~\ref{TabOmegaPlus}. Now assume that $m$ is odd. Then $\gamma\notin\Omega(V)$, and so $G=\mathrm{O}(V)=\Omega_{2m}^+(2).2$. Write $K=\Omega_{2m-2}^-(2).\calO$ with $|\calO|$ dividing $|\N_2^-[G]|/|\Omega_{2m-2}^-(2)|=12$. Since Lemma~\ref{LemOmegaPlus15} implies the existence of a factorization $G=H\N_2^-[G]$ with $H\cap\N_2^-[G]=\SL_{m-2}(2).[6]$, we conclude that $H\cap K$ is a subgroup of $\SL_{m-2}(2).[6]$ of index $|\N_2^-[G]|/|K|=12/|\calO|$. This implies that $|\calO|$ is divisible by $2$, and so $K\geqslant\Omega_{2m-2}^-(2).2$.

Since $\N_2^-[G]=\mathrm{O}_2^-(2)\times\mathrm{O}_{2m-2}^-(2)$, there are exactly three maximal subgroups of index $2$ in $\N_2^-[G]$, say, $M_1$, $M_2$ and $M_3$, and exactly one of them, say $M_1$, is contained in $L$. Let $N=\Omega_2^-(2)\times\Omega_{2m-2}^-(2)<\N_2^-[G]$. Then $N$ is the common normal subgroup of index $2$ in $M_1$, $M_2$ and $M_3$. From Lemma~\ref{LemOmegaPlus15} we derive that $G=HM_1$ with $H\cap M_1=H\cap N=\SL_{m-2}(2).3$. By Lemma~\ref{LemOmegaPlus14}, at least one of $M_2$ or $M_3$, say $M_2$, satisfies $G=HM_2$ with $H\cap M_2=H\cap N=\SL_{m-2}(2).3$. It follows that
\[
\SL_{m-2}(2).3=H\cap N\leqslant H\cap M_3\leqslant H\cap\N_2^-[G]=\SL_{m-2}(2).[6].
\]
Then since $(H\cap M_1)\cup(H\cap M_2)\cup(H\cap M_3)=H\cap(M_1\cup M_2\cup M_3)=H\cap\N_2^-[G]$, we conclude that $H\cap M_3=\SL_{m-2}(2).6$. This implies that $|H\cap M_3|\neq|H||M_3|/|G|$, and so $G\neq HM_3$.

Now that $G=HK$ and $G\neq HM_3$, we obtain $K\nleqslant M_3$. Thus either $K=\N_2^-[G]$, or $K\leqslant M_1$ or $M_2$. Since $K\geqslant\Omega_{2m-2}^-(2).2$ and all involutions in $\mathrm{O}_2^-(2)$ are conjugate, it follows that $(H,K)$ tightly contains the pair $(X,Y)=(\SL_m(2).2,\Omega_{2m-2}^-(2).2)$ in row~8 of Table~\ref{TabOmegaPlus}.

\textsc{Case~3}: $q=4$. It is shown in~\cite[Page~69,~(i)]{LPS1990} that there is no such factorization $G=HK$ for $G=\mathrm{O}(V)$. Hence there is no such factorization $G=HK$ for $G=\Omega(V)$. Therefore, neither $HL$ nor $KL$ is contained in $\mathrm{O}(V)$. Since $\GaO_{2m-2}^-(4)/\Omega_{2m-2}^-(4)$ is cyclic of order $4$, this implies that $(H,K)$ tightly contains the pair $(X,Y)=(\SL_m(4).2,\Omega_{2m-2}^-(4).4)$ in row~10 of Table~\ref{TabOmegaPlus}.
\end{proof}

Corresponding to the candidates $(H^{(\infty)},K^{(\infty)})$ in rows~11--15 of Table~\ref{TabInftyOmegaPlus}, the final lemma of this section deals with the case $K^{(\infty)}=\lefthat\SU_m(q)$.

\begin{lemma}
Suppose that $(H^{(\infty)},K^{(\infty)})$ lies in rows~\emph{11--15} of Table~$\ref{TabInftyOmegaPlus}$. Then $(H,K)$ tightly contains $(X^\alpha,Y^\alpha)$ for some $(X,Y)$ in rows~\emph{11--14 or 23--24} of Table~$\ref{TabOmegaPlus}$ and $\alpha\in\Aut(L)$.
\end{lemma}

\begin{proof}
Since $(H^{(\infty)},K^{(\infty)})$ lies in rows~11--15 of Table~\ref{TabInftyOmegaPlus}, we have $K^{(\infty)}=\lefthat\SU_m(q)$ with $m$ even, and $H^{(\infty)}=q^{2m-2}{:}\Omega_{2m-2}^+(q)$, or $\Omega_{2m-2}^+(q)$ with $q\in\{2,4\}$, or $\Omega_8^-(q^{1/2})$ with $m=4$ and $q\in\{4,16\}$.

\textsc{Case~1}: $H^{(\infty)}=q^{2m-2}{:}\Omega_{2m-2}^+(q)$. In this case, $(H,K)$ tightly contains $(X^\alpha,Y^\alpha)$ for some $(X,Y)$ in row~11 of Table~\ref{TabOmegaPlus} and $\alpha\in\Aut(L)$.

\textsc{Case~2}: $H^{(\infty)}=\Omega_{2m-2}^+(q)$ with $q\in\{2,4\}$. By Tables~\ref{TabMaxOmegaPlus2} and~\ref{TabMaxOmegaPlus1}, replacing $H$ and $K$ by their images under some automorphism of $L$ if necessary, we have $A\cap L=\Pa_1$ or $\N_2^+$. By~\cite[Proposition~2.1]{Sah1977}, $\Pa_1$ has a unique conjugacy class of subgroups isomorphic to $\Omega_{2m-2}^+(q)$. Thus it always holds that $H^{(\infty)}<\N_2^+$, and so $H\leqslant\N_2^+[G]$. Since $K^{(\infty)}=\SU_m(q)$, we have $K\leqslant B<\Omega(V){:}\langle\phi\rangle$ (see~\cite[Table~3.5.E]{KL1990}). Note that $\GU_m(q).(2f)$ has a unique conjugacy class of subgroups isomorphic to $\SU_m(q).(2f)$ as $q\in\{2,4\}$.

First assume that $q=2$ and $H\cap L=\Omega_{2m-2}^+(2)$. Since $K<\Omega(V){:}\langle\phi\rangle=L$, we have $L=(H\cap L)K$. Write $K=\SU_m(2).\calO$ with $|\calO|$ dividing $|B|/|\SU_m(2)|=6$. From Lemma~\ref{LemOmegaPlus19} we deduce that $L=(H\cap L)B$ with $(H\cap L)\cap B=\SU_{m-2}(2).3$. Hence
\[
|(H\cap L)\cap K|=\frac{|(H\cap L)||K|}{|L|}=\frac{|(H\cap L)||B|}{|L|(|B|/|K|)}=\frac{|(H\cap L)\cap B|}{|B|/|K|}=\frac{|\SU_{m-2}(2).3|}{|B|/|K|},
\]
which means that $(H\cap L)\cap K$ is a subgroup of $\SU_{m-2}(2).3$ of index $|B|/|K|=6/|\calO|$. Consequently, $|\calO|=2$ or $6$, and so $K\geqslant\SU_m(2).2$. Since all involutions in $B/\SU_m(2)$ are conjugate, it follows that (up to conjugate in $L$) the pair $(H,K)$ tightly contains the pair $(X,Y)=(\Omega_{2m-2}^+(2),\SU_m(2).2)$ in row~13 of Table~\ref{TabOmegaPlus}.

Next assume that $q=2$ and $H\cap L>\Omega_{2m-2}^+(2)$. Then $H\cap L=\N_2^+[L]$ as $H\cap L\leqslant\N_2^+[L]=\Omega_{2m-2}^+(2).2$. Hence $(H,K)$ tightly contains the pair $(X,Y)=(\N_2^+[L],\SU_m(2))$ in row~12 of Table~\ref{TabOmegaPlus}.

Now assume that $q=4$. By~\cite[Theorem~A]{LPS1990}, such a factorization does not exist for $G=L$. This together with the conclusion $K<\Omega(V){:}\langle\phi\rangle$ implies that $G\geqslant\Omega(V){:}\langle\phi\rangle$ and $\Omega(V){:}\langle\phi\rangle=(H\cap(\Omega(V){:}\langle\phi\rangle))K$ with neither $H\cap(\Omega(V){:}\langle\phi\rangle)$ nor $K$ contained in $L$. Thus $(H,K)$ tightly contains the pair $(X,Y)=(\Omega_{2m-2}^+(4).2,\SU_m(4).4)$ in row~14 of Table~\ref{TabOmegaPlus}.

\textsc{Case~3}: $\Omega_8^-(q^{1/2})$ with $m=4$ and $q\in\{4,16\}$. By~\cite[Remark~1.5]{GGP}, such a factorization only exists for $G=\Omega_8^+(q).f$. Thus $HL=KL=G=\Omega_8^+(q).f$. Since $\GU_4(q).(2f)$ has a unique conjugacy class of subgroups isomorphic to $\SU_4(q).(2f)$, it follows that $(H,K)$ tightly contains $(X^\alpha,Y^\alpha)$ for some $(X,Y)$ in rows~23--24 of Table~\ref{TabOmegaPlus} and $\alpha\in\Aut(L)$.
\end{proof}

\section*{Acknowledgments}
The first author acknowledges the support of NNSFC grants no.~11771200 and no.~11931005. The second author acknowledges the support of NNSFC grant no.~12061083.

\end{document}